\documentclass[10pt,a4paper]{amsart}
\usepackage{amssymb,amsmath,mathrsfs}
\usepackage{bm}
\newtheorem{theorem}{Theorem}[section]
\newtheorem{lemma}[theorem]{Lemma}
\newtheorem{proposition}[theorem]{Proposition}

\newtheorem{definition}[theorem]{Definition}
\newtheorem{remark}[theorem]{Remark}
\numberwithin{equation}{section}

\newcommand{\N}{\mathbb{N}}
\newcommand{\R}{\mathbb{R}}
\newcommand{\Rn}{\mathbb{R}^{n}}

\newcommand{\dif}{\mathrm{d}}

\newcommand{\BBV}{\mathbb{B}_{\mathbb{V}}}
\newcommand{\Dy}{\mathscr{D}}
\newcommand{\epy}{\bm{\varepsilon}}
\newcommand{\BV}{\mathrm{BV}}

\newcommand{\ONE}{\mathbf{1}}
\newcommand{\restrict}{\begin{picture}(10,8)\put(2,0){\line(0,1){7}}\put(1.8,0){\line(1,0){7}}\end{picture}}

\newcommand{\lebn}{\mathscr{L}^{n}}

\newcommand{\lip}{\operatorname{lip}}
\newcommand{\CC}{\operatorname{C}}
\newcommand{\WW}{\operatorname{W}}

\newcommand{\hold}{\operatorname{C}}
\newcommand{\lebe}{\operatorname{L}}
\newcommand{\LIP}{\operatorname{LIP}}

\newcommand{\sobo}{\operatorname{W}}
\newcommand{\bv}{\operatorname{BV}}

\newcommand{\wstar}{\stackrel{*}{\rightharpoonup}}
\newcommand{\toY}{\overset{\mathrm{Y}}{\to}}

\newcommand{\dd}{\mathrm{d}}
\newcommand{\cala}{\mathscr{A}}
\newcommand{\bala}{\mathscr{B}}

\newcommand{\calA}{\mathcal{A}}
\newcommand{\V}{\mathbb{V}}
\newcommand{\W}{\mathbb{W}}
\newcommand{\HH}{\mathbb{H}}

\newcommand{\Meas}{\mathcal{M}^{+}_{1}}
\newcommand{\SV}{\mathbb{S}_{\V}}
\newcommand{\X}{\mathbb{X}}

\newcommand{\U}{\mathbb{U}}
\def\Xint#1{\mathchoice
   {\XXint\displaystyle\textstyle{#1}}%
   {\XXint\textstyle\scriptstyle{#1}}%
   {\XXint\scriptstyle\scriptscriptstyle{#1}}%
   {\XXint\scriptscriptstyle\scriptscriptstyle{#1}}%
   \!\int}
\def\XXint#1#2#3{{\setbox0=\hbox{$#1{#2#3}{\int}$}
     \vcenter{\hbox{$#2#3$}}\kern-.5\wd0}}

\def\dashint{\Xint-}

\usepackage[left=3.3cm,right=3.3cm,top=4cm,bottom=3.5cm]{geometry}

\usepackage{color}

\usepackage[colorlinks=true,citecolor=black,linkcolor=black]{hyperref}

\usepackage{combelow}

\allowdisplaybreaks

\begin{document}
\title[Oscillation and concentration in sequences of
PDE constrained measures]{Oscillation and concentration in sequences of \\
	PDE constrained measures}

\author[J.~Kristensen]{Jan Kristensen}
\author[B.~Rai\cb{t}\u{a}]{Bogdan Rai\cb{t}\u{a}}

\begin{abstract}
We show that for constant rank partial differential operators $\cala$, generalized Young measures generated by sequences
of $\cala$-free measures can be characterized by duality with $\cala$-quasiconvex integrands of 
linear growth.
\end{abstract}

\maketitle

\vspace{-8.5pt}
\section{Introduction}

Since their introduction in \cite{Y}, Young measures have been used to effectively describe oscillation phenomena in non-convex problems in the
calculus of variations and optimal control theory and, later, in non-linear partial differential equations \cite{Tartar1,Tartar3,DiPerna}. In
more general terms \cite{DPM,AlBo,balder,BeLa,Sychev}, one can think of Young measures as wieldy tools to describe the gap between weak and strong convergence in the
interaction with non-linear functionals. More precisely, Young measures provide the one-point statistics for oscillation and (aspects of) concentration
in a sequence and thus allows the computation of its moments.

The interaction of weak convergence and non-linear quantities represents a highly non-trivial problem that can be investigated from many angles.
The classical results in \cite{MorreyB,Reshetnyak,Ball} show that the structure of continuous functionals acting on weakly convergent sequences of
gradients is fairly rigid. The theory of compensated compactness \cite{Tartar1,Tartar2,Tartar3,Murat1,Murat2} was built around the question of
characterizing the non-linear functionals that are weakly continuous when evaluated on sequences that satisfy certain linear partial differential constraints.
One needs only make the elementary observation that these weakly continuous quantities are necessarily polynomials to see that their structure is indeed
very restricted. Moreover, under the constant rank assumption of \cite{Murat2} on the differential constraint, it was shown in \cite{GR} that all such
polynomials are computable and explicit examples show that weakly continuous functionals are very few.

On the other hand, it was already shown in \cite{Tartar1} that the structure of weakly lower semicontinuous quadratic forms is much richer.
Therein, the natural question of characterizing the  functionals that are weakly lower semicontinuous when acting on constrained sequences was asked.
Of interest to us is the comprehensive answer given in \cite{FM99}, where it is shown building on \cite{MorreyB,DaLN,Murat2} that, under the constant rank condition,
weak lower semi-continuity is equivalent with $\cala$-quasiconvexity (see Section~\ref{sec:prel} for notation and terminology). Moreover, the duality via
Jensen-type inequalities of $\cala$-quasiconvex functions of $p$-growth and oscillation Young measures generated by weakly-$\lebe^p$ convergent
$\cala$-free sequences was established, extending the results in \cite{KP1,KP2}.

Similar characterizations can  account for concentration effects, as can be seen from \cite{FMP,FK}. Such extensions rely on decomposition lemmas for
$\lebe^p$-weakly convergent sequences for $1<p<\infty$, which highlight the decoupling of oscillation and concentration effects in this range (see \cite{JK,FMP,Kr}).
At the endpoint $p=1$, concentration effects arising from the application of functionals of linear growth can only be detected for sequences that converge
weakly-* in the sense of measures. In this case, the oscillation and concentration effects of an $\cala$-free sequence may fail to display $\cala$-free
features, as can be seen from the example of a single gradient measure \cite{Alberti_lusin}. Consequently, the lower semicontinuity problem in this case
is substantially more difficult and has been settled in \cite{AmbrosioDalMaso,FM,KR1} in the case of gradient measures and in \cite{ARDPR,BDG} in more general
contexts, building on the recent fundamental results in \cite{DPR,KirKri} (see also \cite{BCMS}, which deals with the case of first order operators $\cala$,
without using a rank-one type theorem \cite{Alberti,DPR} or the automatic convexity result in \cite{KirKri}). 

The characterization of oscillation and concentration in the spirit of \cite{KP1,KP2} was extended to the $\BV$ set-up in \cite{KR2,KR1},
and extended to symmetrized gradient measures in \cite{GDPFR2}. A fairly restricted class of $\cala$-free measures was considered in \cite{BMS}.

The purpose of the present paper is two-fold. We provide a general characterization of $\cala$-free generalized Young measures by duality with $\cala$-quasiconvex
functions of linear growth via Jensen-type inequalities under the assumptions that the operator $\cala$ has constant rank and that its wave cone is spanning (for the definitions of these assumptions, as well as the terminology used in the statement below, see Sections~\ref{sec:prel},~\ref{sec:ym}). However, perhaps
more importantly, our purpose is also expository: when restricting our proof to the basic gradient case, the argument is significantly shorter and more streamlined than
the existing ones. The same can be said about the case without concentration effects (cf. \cite{KP2,JK,Kr,FM99}).

Recall from \cite{Rpot} that there exists another differential operator $\bala$ such that $\cala$-quasiconvexity is equivalent with
$\cala$-$\bala$-quasiconvexity in the sense of \cite{DaLN,DaFo}; in particular, $\cala\circ\bala\equiv 0$. Say that the order of $\bala$ is $l$ and
let $\Omega\subset\R^n$ be a bounded open set such that $\mathscr{L}^n(\partial\Omega)=0$.
\begin{theorem}\label{main}
  Let $\nu$ be a Young measure on $\Omega$ such  that $\lambda(\partial\Omega)=0$ and let $v\in\mathcal{M}(\Omega,\mathbb V)$ be its barycentre. Then there exist a sequence of mollifiers
  $\phi_j\wstar\delta_0$ in $\mathcal M(\Omega)$ and a sequence of maps $u_j\in\hold^\infty_c(\Omega,\mathbb{U})$ such that
$$
\begin{cases}
\phi_j*v+\bala u_j\toY \nu\\
\|u_j\|_{\sobo^{l-1,1}}\rightarrow 0
\end{cases}
$$
if and only if $\cala v=0$ and for all $\cala$-quasiconvex $f\colon\mathbb{V}\rightarrow\R$ of linear growth it holds that
$$
\int_{\mathbb{V}}f\dif \nu_x+\lambda^a(x)\int_{\mathbb{S}_\mathbb{V}}f^\infty\dif\nu_x^\infty\geq
f(\overline{\nu}_x+\lambda^a(x)\overline{\nu}_x^\infty)\quad\text{for $\lebn$-a.e. }x\in\Omega,
$$
where $\lambda=\lambda^a\lebn\restrict\Omega+\lambda^s$, $\lambda^s\perp \lebn$, is a Radon--Nikod\'ym decomposition of $\lambda$.
\end{theorem}
Our argument to prove this follows the broad two-step structure that several previous proofs have, namely to first deal with the case of homogeneous
Young measures and then employ an approximation procedure. The main conceptual novelty of this work is that our approximation procedure is completely
measure theoretic, as can be seen from Section~\ref{sec:inhom}. In particular, we do not use any fine structure properties of $\cala$-free measures;
these are only used in the homogeneous step, see Section~\ref{sec:HB}.

A preliminary version of our result, covering the case when the barycentre $v$ has no singular part, appeared in \cite[Chapter~3]{Rthesis}.
As we show here, the strategy used there is flexible enough to also deal with the interaction between the singular parts of the barycentre and
concentration measures, as we implement in Lemmas~\ref{reggen} and \ref{singgen}. We record that
a similar result is asserted in the recent preprint \cite{Adolf}; 
{there, the Helmholtz-type decomposition from \cite{FM99} plays a key role, whereas here we rely on local potentials as in \cite{Rpot}.}

Finally, in Section~\ref{sec:diff_conc} we will focus on the (diffuse) concentration  angle measure. Recall that as a consequence of the main results in \cite{DPR,KirKri}, the measures $\nu_x^\infty $ are unconstrained for $\lambda^s$ almost every $x$. Here, we will present a new necessary Jensen-type inequality that $\nu_x^\infty$ satisfies $\lambda^a\lebn$ almost everywhere if $\nu$ is $\cala$-free. Under a decoupling condition, we then show that this inequality is also sufficient to generate $\nu$ with $\cala$-free fields.

This paper is organized as follows: In Section~\ref{sec:prel} we recall some properties of the Kantorovich norm, linear partial differential operators,
and quasiconvex and directionally convex functions. In Section~\ref{sec:proof} we introduce Young measures with some detail on the functional analytic
aspects, as well as more technical properties. In Section~\ref{sec:HB} we essentially prove Theorem~\ref{main} in the case of homogeneous Young measures,
whereas in Section~\ref{sec:inhom} we perform the crucial approximation argument. In Section~\ref{sec:diff_conc} we comment on the diffuse concentration of $\cala$-free Young measures.

\section{Preliminaries}\label{sec:prel}

\subsection{Kantorovich norm}
Let $\mathrm{X}$ be a compact and separable metric space. We recall
that $\CC (\mathrm{X})$ with the supremum norm is a separable Banach space
whose dual can be identified with the space $\mathcal{M}( \mathrm{X})$ of signed bounded Radon measures on
$\mathrm{X}$. The subspace $\LIP (\mathrm{X})$ consisting of all Lipschitz
functions $\Phi \colon \mathrm{X} \to \R$ is a (non-separable) Banach space under the norm
$$
\| \Phi \|_{\LIP} = \sup_{x \in \mathrm{X}} | \Phi (x)| + \lip (\Phi ).
$$
The Kantorovich norm of a signed bounded Radon measure $\mu$ on $\mathrm{X}$ is here defined as
$$
\| \mu \|_{\mathrm{K}} = \sup \biggl\{ \langle \mu , \Phi \rangle : \, \Phi \in \LIP ( \mathrm{X}), \, \| \Phi \|_{\LIP} \leq 1 \biggr\} ,
$$
so it is the dual norm of $\| \cdot \|_{\LIP}$ restricted to $\mathcal{M}( \mathrm{X})$. 
We shall be interested in its restriction to the space $\mathcal{M}^{+}(\mathrm{X})$ of positive bounded Radon measures
that becomes a metric space under the Kantorovich metric $\mathrm{d}_{\mathrm{K}}(\mu, \nu ) = \| \mu - \nu \|_{\mathrm{K}}$. We record
the useful fact that for $\mu \in \mathcal{M}^{+}( \mathrm{X})$ we have
\begin{equation}\label{kantnorm}
\| \mu \|_{\mathrm{K}} = \mu ( \mathrm{X}).
\end{equation}

\begin{proof}
Since $\ONE_{\mathrm{X}} \in \LIP (\mathrm{X})$ and $\| \ONE_{\mathrm{X}} \|_{\LIP} =1$ we clearly have $\| \mu \|_{\mathrm{K}} \geq \mu (\mathrm{X})$.
Conversely, if $\Phi \in \LIP (\mathrm{X})$ with $\| \Phi \|_{\LIP} \leq 1$, then in particular $\max_{\mathrm{X}} | \Phi | \leq 1$ so
using that $\mu$ is a positive measure we get
$$
\langle \mu , \Phi \rangle \leq \langle \mu , \ONE_{\mathrm{X}} \rangle = \mu (\mathrm{X}),
$$
hence taking supremum over such $\Phi$ we arrive at the opposite inequality.
\end{proof}
As a subset of $\CC (\mathrm{X})^{\ast}$ the space $\mathcal{M}^{+}(\mathrm{X})$ also inherits the weak$\mbox{}^{\ast}$ topology:
$$
\biggl\{ \mathcal{O} \cap \mathcal{M}^{+}(\mathrm{X}) : \, \mathcal{O} \in \sigma \bigl( \CC (\mathrm{X})^{\ast},\CC (\mathrm{X}) \bigr) \biggr\} .
$$
\begin{lemma}\label{weakstarkanto}
On $\mathcal{M}^{+}(\mathrm{X})$ the relative weak$\mbox{}^{\ast}$ topology is exactly the topology determined by the
Kantorovich metric $\mathrm{d}_{\mathrm{K}}$.
\end{lemma}

\begin{proof}
In order to show that $\mathcal{M}^{+}(\mathrm{X}) \cap \mathcal{O}$ is open relative to the Kantorovich metric for each
weak$\mbox{}^{\ast}$ open $\mathcal{O}$ it suffices to show that for each $\Phi \in \CC (\mathrm{X})$ and $t \in \R$ the set
\begin{equation}\label{basicweakstar}
\biggl\{ \mu \in \mathcal{M}^{+}(\mathrm{X}) : \, \langle \mu , \Phi \rangle < t \biggr\}
\end{equation}
is open relative to the Kantorovich metric on $\mathcal{M}^{+}( \mathrm{X})$. To that end we fix $\mu_0 \in \mathcal{M}^{+}( \mathrm{X})$
with $\langle \mu_{0} , \Phi \rangle < t$. Next, we employ a standard approximation scheme and put for each $j \in \N$,
$$
\Phi_{j}(x) = \sup \bigl\{ \Phi (y)-jd_{\mathrm{X}}(x,y) : \, y \in \mathrm{X} \bigr\} .
$$
Hereby $\lip (\Phi_{j}) \leq j$ and $\Phi_{j}(x) \searrow \Phi (x)$ pointwise in $x \in \mathrm{X}$ (hence uniformly) as $j \nearrow \infty$.
Select $j \in \N$ so
$$
\langle \mu_{0}, \Phi_{j} \rangle <   \frac{t + \langle \mu_{0},\Phi \rangle}{2}.
$$
If we take any positive $r < (t-\langle \mu_{0} , \Phi \rangle )/2j$, then we have for each $\mu \in \mathcal{M}^{+}(\mathrm{X})$ with
$\| \mu - \mu_{0} \|_{\mathrm{K}} < r$ that
\begin{eqnarray*}
\langle \mu , \Phi \rangle &\leq& \langle \mu , \Phi_{j} \rangle\\
&\leq& \langle \mu-\mu_{0},\Phi_{j} \rangle + \frac{t+\langle \mu_{0},\Phi \rangle}{2}\\
&<& r.
\end{eqnarray*}
It follows that the set defined at (\ref{basicweakstar}) is open relative to $\mathrm{d}_{\mathrm{K}}$.

For the opposite inclusion we fix $\mu_{0} \in \mathcal{M}^{+}(\mathrm{X})$ and $r > 0$ and must show that the open ball
$\bigl\{ \mu \in \mathcal{M}^{+}(\mathrm{X}) : \, \| \mu - \mu_{0} \|_{\mathrm{K}} < r \bigr\}$ is weak$\mbox{}^{\ast}$ open relative to
$\mathcal{M}^{+}(\mathrm{X})$. To get started we note that by the Arzela-Ascoli Theorem 
the set $\mathbb{S}=\bigl\{ \Phi : \, \| \Phi \|_{\LIP} =1 \bigr\}$ is totally bounded in $\CC (\mathrm{X})$. Hence for each $\delta > 0$ we
can find a finite $\delta$-net $\Delta$ in $\mathbb{S}$. Without loss in generality we may assume that $\ONE_{\mathrm{X}} \in \Delta$.
Now put, for a $t \in \R$ to be specified,
$$
H_{\Delta} = \biggl\{ \mu \in \mathcal{M}^{+}(\mathrm{X}) : \, \max_{\Phi \in \Delta} \langle \mu-\mu_{0}, \Phi \rangle < t \biggr\} 
$$
and note that for $\mu \in H_{\Delta}$ and $\Psi \in \mathbb{S}$ we have
\begin{eqnarray*}
\langle \mu - \mu_{0}, \Psi \rangle &=& \min_{\Phi \in \Delta} \biggl( \langle \mu - \mu_{0},\Phi \rangle + \langle \mu - \mu_{0},\Psi - \Phi \rangle \biggr)\\
&<& t+\bigl( \mu (\mathrm{X}) + \mu_{0}( \mathrm{X}) \bigr)\delta .
\end{eqnarray*}
In order to bound the last term we use that $\ONE_{\mathrm{X}} \in \Delta$. It entails that $\mu ( \mathrm{X}) < t + \mu_{0}( \mathrm{X})$ and consequently
$$
\langle \mu - \mu_{0}, \Psi \rangle < t + \bigl( 2\mu_{0}( \mathrm{X}) +t \bigr) \delta .
$$
We leave it to the reader to check that we may choose
$$
t= \frac{r}{2} \, \mbox{ and then any } \, \delta \in (0,\frac{r}{4 \mu_{0}( \mathrm{X}) +r})
$$
to complete the proof.
\end{proof}
\subsection{Linear partial differential operators}
We will work with linear homogeneous partial differential operators on $\R^n$
\begin{align}\label{pdos}
    \cala\equiv\sum_{|\alpha|=k}A_\alpha\partial^\alpha \quad \text{and}\quad \bala\equiv\sum_{|\beta|=l}B_\beta\partial^\beta,
\end{align}
where $A_\alpha\in \mathscr{L}(\V.\W)$ and $B_\beta\in \mathscr{L}(\U, \V)$ for all $\alpha,\beta\in \mathbb{N}_0^n$ with $|\alpha|=k$ and $|\beta|=l$. We write for $\xi\in\R^n$
$$
\cala(\xi)\equiv\sum_{|\alpha|=k}\xi^\alpha A_\alpha \quad \text{and}\quad \bala(\xi)\equiv\sum_{|\beta|=l}\xi^\beta B_\beta.
$$
Throughout we will assume that $\cala$ has constant rank, meaning that $\mathrm{rank}\cala(\xi)$ is independent of $\xi\neq 0$. Moreover, this assumption is equivalent with the existence of $\bala$ as above such that
$$
\ker\cala (\xi)=\mathrm{im\,}\bala(\xi)\quad \text{for }\xi\neq 0,
$$
see \cite{Rpot}. The existence of such $\bala$ is assumed implicitly in what follows.

We also record the definition of the wave cone of $\cala$,
$$
\Lambda_\cala\equiv \bigcup_{\xi\in\mathbb{S}^{n-1}}\ker\cala(\xi).
$$

Our core object of study will be $\cala$-free measures, for which we record the following remarkable structure theorem from {\cite{DPR}}:

\begin{lemma}\label{rankonetheorem}
If $v \in \mathcal{M}( \Omega , \V )$ satisfies $\cala v = 0$ in the sense of distributions, then
$$
\frac{\dd v^s}{\dd |v^{s}|} \in \Lambda_{\cala} \quad |v^{s}|\mbox{-almost everywhere in } \Omega .
$$
\end{lemma}

\subsection{$\cala$-quasiconvex integrands}

We recall from \cite[Definition~3.1]{FM99} that an integrand $f \colon \V \to \R$ is said to be $\cala$-quasiconvex if 
\begin{align*}
f(z)\leq\int_\X f(z+v(x))\dif x
\end{align*}
holds for all $z \in\V$ and all $v\in\hold^\infty_{\mathrm{per}}(\X,\V)$ such that $\int_\X v\dif x=0$ and $\cala v=0$.
We wrote $\X$ for the unit cube $(-\tfrac{1}{2},\tfrac{1}{2})^n$. If $\Phi\in\hold(\V)$ may fail to be $\cala$-qusiconvex, then one defines
the $\cala$-quasiconvex envelope $\Phi^{\mathrm{qc}}$ of $\Phi$ as
\begin{align*}
\Phi^{\mathrm{qc}}(\xi)=\inf\left\{\int_\X\Phi(\xi+v(x))\dif x\colon v\in\hold^\infty_{\mathrm{per}}(\X,\V),\,\int_\X v=0,\,\cala v=0 \right\}.
\end{align*}
\begin{lemma}[{\cite[Proposition~3.4]{FM99}}]\label{qcrc}
If $f \in \hold(\V)$, then $f^{\mathrm{qc}}$ is $\cala$-quasiconvex. If $f$ is $\cala$-quasiconvex, then it is also  $\Lambda_{\cala}$-convex. 
\end{lemma}
We have the following elementary property for integrands of linear growth:
\begin{lemma}\label{lem:finite_valued}
Suppose that $\cala$ is an operator of the form \eqref{pdos} for which the wave cone is spanning, i.e.,
$\mathrm{span\,}\Lambda_{\cala} =\V$. Let $\Phi\in\hold(\V)$ have linear growth $|\Phi|\leq c(|\cdot|+1)$ and be such that its
$\cala$-quasiconvex envelope satisfies
\begin{align*}
\Phi^{\mathrm{qc}}(\xi_0)>-\infty\quad\text{ for some }\xi_0\in\V.
\end{align*}
Then $\Phi^{\mathrm{qc}}$ is real valued, $\cala$-quasiconvex and has linear growth, $|\Phi^{\mathrm{qc}}|\leq c(|\cdot|+1)$.
\end{lemma}
\begin{proof}
By Lemma~\ref{qcrc}, we have that $\Phi^{\mathrm{qc}}$ is convex in the directions of $\Lambda_{\cala}$.
In particular, whenever $\Phi^{\mathrm{qc}}(\eta)>-\infty$, it is also the case that $\Phi^{\mathrm{qc}}>-\infty$ on $\eta+\Lambda_{\cala}$.
Now let $\xi\in\V$. Since $\Lambda_{\cala}$ spans $\V$, we can find vectors $v_1,\ldots,v_s\in\Lambda_{\cala}$ such that $\xi-\xi_0=\sum_{i=1}^s v_i$.
By the previous consideration, it follows that $\Phi^{\mathrm{qc}}(\xi)>-\infty$, hence $\Phi^{\mathrm{qc}}$ is real valued. 
	
We choose a basis $\{e_i\}_{i=1}^d\subset\Lambda_{\cala}$ of $\V$, where $d=\dim\V$.
Let $M$ be a change of basis matrix such that $\{M^{-1}e_i\}_{i=1}^d$ is an orthonormal basis of $\V$.
Then $\Phi^{\mathrm{qc}}\circ M$ is separately convex. Using the facts that norms are equivalent on finite dimensional
spaces and $\Phi^{\mathrm{qc}}\leq \Phi$, we have that 
\begin{align*}
\Phi^{\mathrm{qc}}\circ M\leq c(M)(|\cdot|_M+1),
\end{align*}
where $|\cdot|_M$ denotes the Euclidean norm in the new coordinate system. By \cite[Lemma~2.5]{Kr}, we have that
$\Phi^{\mathrm{qc}}\circ M$ has linear growth from below. Again by norm equivalence, we have that 
\begin{align*}
\Phi^{\mathrm{qc}}\geq -c(M)(|\cdot|+1).
\end{align*}
The fact that $\Phi^{\mathrm{qc}}$ is $\cala$-quasiconvex follows from Lemma~\ref{qcrc}.
\end{proof}
We also record the following equivalent definition of $\cala$-quasiconvex envelopes in terms of $\bala$, which will be used
to prove the characterization of homogeneous $\cala$-free Young measures:
\begin{lemma}[{\cite[Corollary 5]{Rpot}}]
Let $\Phi\in\hold(\V)$, $\X=(0,1)^n$, and $\varepsilon>0$. Then 
\begin{align}\label{eq:relax_formula_YM}
\Phi^{\mathrm{qc}}(z)=\inf\left\{\int_\X\Phi( z+\bala u)\colon  u\in\hold^\infty_c(\X,\U),\,\|D^{l-1} u\|_{\lebe^\infty}<\varepsilon\right\}
\end{align}
for all $z \in\V$.
\end{lemma}

Finally, we also record a consequence of the main result of \cite{KirKri}:
\begin{lemma}\label{1homo}
  Let $f \colon \V \to \R$ be a positively $1$-homogeneous and $\Lambda_{\cala}$-convex integrand, where $\Lambda_\cala$ is spanning. Then $f$ is convex at each point
  of $\Lambda_{\cala}$.
\end{lemma}

\section{Proof of main result}\label{sec:proof}

\subsection{Functional set-up}\label{sec:hom}
Let $\HH$ be the space of continuous integrands $\Phi \colon \V \to \R$ of linear growth that
admit a regular recession integrand:
$$
\Phi^{\infty} (z) = \lim_{t \to \infty} \frac{\Phi (tz)}{t} \quad \mbox{exists locally uniformly in } z \in \V .
$$
We recall that this is exactly the class of integrands $\Phi \colon \V \to \R$ for which the transformed integrand
$$
\hat{z} \mapsto \bigl( 1-| \hat{z}| \bigr) \Phi \left( \frac{\hat{z}}{1-|\hat{z}|} \right)
$$
is bounded and uniformly continuous on the open unit ball $\BBV (0,1)$. We endow $\HH$ with the norm
$$
\| \Phi \| = \sup_{z \in \V} \frac{|\Phi (z)|}{1+|z|}.
$$
The mapping $T \colon \HH \to \CC (\overline\BBV )$ defined by
$$
T\Phi ( \hat{z}) = \left\{
\begin{array}{ll}
\bigl( 1-| \hat{z}| \bigr) \Phi \left( \frac{\hat{z}}{1-|\hat{z}|} \right) & \mbox{ if } | \hat{z}| < 1\\
\Phi^{\infty} ( \hat{z} ) & \mbox{ if } | \hat{z}| = 1,
\end{array}
\right.
$$
is then easily seen to be an isometric isomorphism. Hence so is the dual mapping $T^\ast \colon \CC (\overline\BBV )^\ast \to \HH^\ast$.
By virtue of the Riesz representation theorem we may identify the dual space $\CC (\overline \BBV )^{\ast}$ with the space of bounded signed
Radon measures $\mathcal{M}( \overline\BBV )$. Hence given $\ell \in \HH^\ast$ there is a unique $\mu \in \mathcal{M} (\overline\BBV )$ so
$T^{\ast}\mu = \ell$. Consequently we have for $\Phi \in \HH$,
\begin{eqnarray*}
  \ell ( \Phi ) &=& \langle T^{\ast}\mu , \Phi \rangle = \int_{\BBV} \! T\Phi \, \dd \mu\\
  &=& \int_{\BBV (0,1)} \! \bigl( 1-| \hat{z}| \bigr) \Phi \left( \frac{\hat{z}}{1-| \hat{z}|} \right) \, \dd \mu ( \hat{z}) +\int_{\SV} \! \Phi^{\infty} \, \dd \mu .
\end{eqnarray*}
If we put $S(\hat{z}) = \tfrac{\hat{z}}{1-| \hat{z}|}$ for $\hat{z} \in \BBV (0,1)$ and
$$
\begin{array}{l}
  \mu^{0} = S_{\#}\biggl[ \bigl( 1-| \hat{z}| \bigr) \mu \restrict \BBV (0,1) \biggr] ,\\
  \mu^{\infty} = \mu \restrict \SV ,
\end{array}
$$
then $\mu^{0} \in \mathcal{M}( \V )$ has total variation with finite first moment, that is, $\langle | \mu^0 | , | \cdot | \rangle < \infty$,
$\mu^{\infty} \in \mathcal{M}( \SV )$, and
\begin{equation}\label{hdual}
\ell ( \Phi ) = \int_{\V} \! \Phi \, \dif \mu^0 + \int_{\SV} \! \Phi^{\infty} \, \dif \mu^{\infty} .
\end{equation}
In view of this representation we identify in the following each $\ell \in \HH^\ast$ with the unique pair of measures $( \mu^{0}, \mu^{\infty})$
as above. We also record that
$$
\| \ell \| = \int_{\V} \! \bigl( 1 + | \cdot | \bigr) \, \dd | \mu^{0}| + | \mu^{\infty}|( \SV )
$$
and that $\ell \geq 0$ when $\mu^{0} \geq 0$, $\mu^{\infty} \geq 0$, so that $T^\ast$ is a positive operator. This means that $T^\ast$ is a homeomorphism
of the positive cones, $\mathcal{M}^{+}( \BBV )$ onto $(\HH^{\ast})^+$ with their respective relative weak$\mbox{}^\ast$ topologies, and consequently,
by virtue of Lemma \ref{weakstarkanto}, that the relative weak$\mbox{}^\ast$ topology on $(\HH^{\ast})^{+}$ is determined by the (push-forward)
Kantorovich metric defined as
$$
\mathrm{d}_{\mathrm{K}}\bigl( (\mu^{0},\mu^{\infty}) , (\nu^{0}, \nu^{\infty}) \bigr) = \| (\mu^{0},\mu^{\infty}) - (\nu^{0}, \nu^{\infty}) \|_{\mathrm{K}}
$$
for $(\mu^{0},\mu^{\infty})$, $(\nu^{0}, \nu^{\infty}) \in \HH^{\ast}$, where the (push-forward) Kantorovich norm is
\begin{equation}\label{kantonh}
\| (\mu^{0},\mu^{\infty}) \|_{\mathrm{K}} = \sup_{\Phi \in \HH , \, \| T\Phi \|_{\LIP} \leq 1} \left| \int_{\V} \! \Phi \, \dd \mu^{0} + \int_{\SV} \! \Phi^{\infty} \, \dd \mu^{\infty} \right| .
\end{equation}
If $v$ is a bounded $\V$-valued Radon measure on $\X=(0,1)^n$ with Lebesgue--Radon--Nikod\'{y}m decomposition $v= v^{a}\dif x + v^s$, then we define
$\epy_v \in \HH^\ast$ by
\begin{align}\label{eq:elem_YM}
\epy_{v}( \Phi ) = \int_{\X} \! \Phi (v^{a}(x)) \, \dif x + \int_{\X} \! \Phi^{\infty}(v^s ) \quad (\Phi \in \HH ).
\end{align}
We record a by now standard approximation result for integrands of linear growth:
\begin{lemma}\label{approximationlemma}
Let $f \colon \V \to \R$ be of linear growth and $\cala$-quasiconvex. Then there exist $f_{j} \in \HH$ that are $\cala$-quasiconvex and satisfy
$$
f_{j}(z) \geq f_{j+1}(z) \searrow f(z) \mbox{ as } j \nearrow \infty
$$
for each $z \in \V$.
\end{lemma}
A proof in the case $\cala=\mathrm{curl}$ and which extends immediately  can be found in {\cite[Lemma~6.3]{KirKri}}. Although we will not use this fact, we remark that each integrand $f_j$ thus constructed is positively homogeneous of degree one outside a large ball $B(0,R_j)$.

We conclude the section with an elementary application of the three-slope inequality:

\begin{lemma}\label{rcoflineargrowth}
Let $\Lambda$ be a balanced and spanning cone in $\V$. If $f \colon \V \to \R$ is $\Lambda$-convex and of linear growth, then $f$ is
Lipschitz. Furthermore, for all $z \in \V$ and $w \in \Lambda$ we have
$$
f(z+w) \leq f(z)+f^{\infty}(w).
$$
\end{lemma}
\subsection{Young measures}\label{sec:ym} 
Similarly to the autonomous integrands in Section~\ref{sec:hom}, and following \cite{AlBo,KR2}, we will define the space $\mathbb{E}(\Omega,\V)$ of (non-autonomous)
integrands $\Phi\in\hold(\Omega\times \V)$ of linear growth $|\Phi(x,z)|\leq c(1+|z|)$, where the best constant defines the norm, and
such that the recession integrand
$$
\Phi^\infty(x,z)=\lim_{t\rightarrow\infty,x^\prime\rightarrow x}\frac{\Phi(x^\prime,tz)}{t}
$$
exists uniformly in $(x,z)\in \overline \Omega\times \mathbb{S_V}$. We extend the definition of
$T\colon\mathbb{E}(\Omega,\V)\rightarrow \hold(\overline{\Omega\times \mathbb{B_V}})$ in an obvious way, thereby obtaining an isometric isomorphism.
Consequently, the dual space $\mathbb{E}(\Omega,\V)^*$ can be identified with $\mathcal{M}(\overline{\Omega}\times \overline{\mathbb{B_V}})$.
By the Stone--Weierstrass theorem we have the following:
\begin{lemma}\label{lem:identify}
  Suppose that $\mu\in \mathcal{M}(\overline{\Omega}\times \overline{\mathbb{B_V}})$ is such that $\langle \mu ,\eta\otimes\Psi\rangle=0$
  for all $\eta\in \mathrm{LIP}(\overline{\Omega})$ and $\Psi\in\mathrm{LIP}(\overline{\mathbb{B_V}})$. Then $\mu\equiv 0$.
  Consequently, to identify an element in $\mathbb{E}(\Omega,\V)^*$, it suffices to test only with $\eta\otimes \Phi$,
  where $\|\eta\|_{\mathrm{LIP}(\Omega)}\leq 1$ and $\|T\Phi\|_{\LIP(\mathbb{B_V})}\leq 1$.
\end{lemma}
We can now define Young measures as follows:
\begin{definition}
We say that a triple $\nu=\left((\nu_x)_{x\in\Omega},\lambda,(\nu_x^\infty)_{x\in\overline{\Omega}}\right)$ is a \emph{Young measure} if
\begin{enumerate}
    \item $\nu_x\in\mathcal{M}^+_1(\V)$ for $\lebn$-a.e. $x\in\Omega$ and $x\mapsto \nu_x$ is weakly* Lebesgue measurable;
    \item $\lambda\in\mathcal{M}^+(\overline{\Omega})$;
    \item $\nu_x^\infty\in\mathcal{M}^+_1(\mathbb{S_V})$ for $\lambda$-a.e. $x\in\Omega$ and $x\mapsto \nu_x$ is weakly* $\lambda$ measurable;
    \item $\int_{\Omega}\int_\V |z|\dif\nu_x(z)\dif x<\infty $.
\end{enumerate}
We say that $(\nu_x)_{x\in\Omega}$ is the \emph{oscillation measure}, $\lambda$ is the \emph{concentration measure}, and
$(\nu_x^\infty)_{x\in\overline\Omega}$ is the \emph{concentration-angle measure}.
\end{definition}
We identify Young measures with elements of $\mathbb{E}(\Omega,\V)^*$ via the identity
$$
\langle \nu,\Phi\rangle_{\mathbb{E}^*,\mathbb{E}}\equiv \int_{\Omega}\int_{\V}\Phi(x,z)\dif\nu_x(z)\dif x+
\int_{\overline{\Omega}}\int_{\mathbb{S_V}}\Phi^\infty(x,z)\dif\nu_x^\infty(z)\dif\lambda(x)\quad \text{for }\Phi\in\mathbb{E}(\Omega,\mathbb{V}).
$$
We further identify measures $v\in\mathcal{M}(\Omega,\V)$ with (elementary) Young measures by 
$$
\langle \epy_v,\Phi\rangle_{\mathbb{E}^*,\mathbb{E}}\equiv \int_{\Omega}\Phi(x,v^a(x))\dif x+
\int_{\Omega} \! \Phi^{\infty}\left(x,\dfrac{\dif v^s}{\dif|v^{s}|}(x)\right)\dif|v^{s}|(x)\quad \text{for }\Phi\in\mathbb{E}(\Omega,\mathbb{V}),
$$
where $|v^{s}|$ denotes the total variation measure and $v=v^a\lebn\restrict\Omega +v^s$ is the Lebesgue--Radon--Nikod\'ym decomposition of the measure $v$.
 Finally, we say that a sequence $v_j\in\mathcal{M}(\Omega,\V)$ generates $\nu$ if $\epy_{v_j}\wstar \nu$ in $\mathbb{E}(\Omega,\V)^*$.
\begin{proposition}[{\cite[Proposition~2.24]{ARDPR}}]\label{local}
Assume $v_j$, $v$ are $\V$-valued Radon measures on $\Omega$, $Q\subset\R^n$ is a cube naturally identified with the torus $\mathbb{T}_n$, and $\cala$ has constant rank.
Let $v_j \wstar v$ and $\cala v_j \to 0$ in
$\sobo^{-k,q}( \Omega , \W )$ for some $q \in (1,\tfrac{n}{n-1})$. If $v_j \toY \nu$, then there exist an $\lebn$-null set
$N^a \subset \Omega$ 
such
for each $x \in \Omega \setminus N^a$ the Young measure 
$$
\nu (x) = \left( ( \nu_x )_{y \in Q}, \frac{\dif \lambda}{\dif \lebn}(x) \lebn \restrict Q , ( \nu_{x}^{\infty} )_{y \in Q} \right)
$$
is generated by a sequence $( \tilde v_j ) \subset \CC^{\infty} ( \mathbb{T}_n , \V )$ satisfying 
$$
\tilde v_{j}\lebn \restrict Q \wstar \left( \overline{\nu}_x + \overline{\nu}_{x}^{\infty}\frac{\dif \lambda}{\dif \lebn}(x) \right) \lebn \restrict Q 
\quad \mbox{ and } \quad \cala\tilde  v_{j} = 0. 
$$
\end{proposition}

\subsection{The Hahn--Banach argument}\label{sec:HB}
Fix $z \in \V$ and denote by $Y$ the set of all pairs $(\nu^{0}, \nu^{\infty} ) \in \Meas ( \V ) \times \mathcal{M}^{+}( \SV )$
for which we can find a sequence $(  u_{j} )$ in $\CC^{\infty}_{c}( \X,\U )$ such that, as $j \to \infty$, $\|  u_j \|_{\sobo^{l-1,1}} \to 0$ and
\begin{equation}\label{homgen}
\int_{\X} \! \Phi ( z+ \bala u_{j}(x)) \, \dif x \to \int_{\V} \! \Phi \, \dif \nu^{0} + \int_{\SV} \! \Phi^{\infty} \, \dif \nu^{\infty}
\quad \forall \, \Phi \in \HH .
\end{equation}
Obviously, we can consider $Y$ as a subset of $(\HH^{\ast})^+$ and we clarify that $\mathcal{M}^+(\mathbb{S}_\V)$ denotes the set of positive Radon measures on
$\mathbb{S}_\V$, whereas $\mathcal{M}_1^+(\V)$ denotes the set of probability measures on $\V$. The above definition (\ref{homgen}) of $Y$ can therefore also
be stated in terms of the Kantorovich metric: $\| \epy_{z +\bala u_j}-(\nu^{0},\nu^{\infty}) \|_{\mathrm{K}} \to 0$ as $j \to \infty$.

\begin{remark}\label{bettergenerator}
Using a standard exhaustion argument it is not difficult to show that each $\nu \in Y$ can be generated by sequences with apparently much better
convergence properties: Let $\nu = ( \nu^{0},\nu^{\infty}) \in Y$ and let $\omega \colon [0, \infty ) \to [0, \infty )$ be a sublinear modulus of
continuity (so $\omega (0)=0$, $\omega$ is continuous, increasing, concave and $\omega (t)/t \to \infty$ as $t \searrow 0$). Then there exists
a sequence $( \varphi_{j} )$ in $\CC^{\infty}_{c}( \X , \U )$ so $\| \varphi_{j} \|_{\WW^{l-1,\infty}} \to 0$, $\| \nu - \epy_{z + \bala \varphi_{j}} \|_{\mathrm{K}} \to 0$ and
$$
\sup_{\stackrel{x, y \in \X}{x \neq y}} \frac{\bigl| \nabla^{l-1}\varphi_{j}(x)-\nabla^{l-1}\varphi_{j}(y)\bigr|}{\omega \bigl( |x-y| \bigr)} \to 0.
$$
\end{remark}

\begin{lemma}\label{test}
The family $\bigl\{ \epy_{z+ \bala  u} : \,  u \in \CC^{\infty}_{c}( \X , \U) \bigr\}$ is a weakly-* dense subset of $Y$.
\end{lemma}

\begin{proof}
It suffices to show that $\epy_{z+\bala u} \in Y$ for each $ u \in \CC^{\infty}_{c}( \X , \U)$. But this is a consequence of a standard
exhaustion argument that runs as follows: We extend $ u$ to $\R^n$ by zero without changing notation and subdivide each side of $\X$ into
$j\in\N$ disjoint congruent open intervals, and consider the resulting
mesh of $j^n$ disjoint congruent cubes. We fix an arrangement of the cubes, say
$$
\X_i = x_i + j^{-1}\X , \quad i \in \{ 1, \, \dots \, , \, j^n \} ,
$$
and define
$$
\phi_j(x)=\sum_{i=1}^{j^n}j^{-l} u(j(x-x_i)).
$$
It is then simple to compute that
\begin{align*}
\int_{\X}\Phi(z+\bala u)\dif x&=\int_{\X} \Phi(z+\bala \phi_{j})\dif x\\
\|D^{l-1}\phi_{j}\|_{\lebe^{1}(\X)}&=j^{-1}\|D^{l-1}u\|_{\lebe^{1}(\X)}\rightarrow 0\quad\text{as }j\rightarrow\infty.
\end{align*}
The conclusion follows by applying Poincar\'e's inequality iteratively.
\end{proof}

\begin{lemma}\label{closed}
The set $Y$ is weakly-$\mbox{}^\ast$ closed in $\HH^\ast$.
\end{lemma}

\begin{proof}
Let $\mu = (\mu^0, \mu^\infty ) \in \overline{Y}^{w^\ast}$. Let $\{ \Phi_j \}_{j \in \N}$ be a countable and dense subset of $\HH$ and put $\Phi_0 = | \cdot |$.
Now for each $i \in \N$ we may select $ u_i \in \CC^{\infty}_{c}( \X, \U )$ such that
\begin{equation}\label{closed1}
\|  u_{i} \|_{\sobo^{l-1,1}} < \frac{1}{i} \quad , \quad \max_{0 \leq j \leq i} \left| \int_{\X} \! \Phi_{j} (z + \bala u_{i}) \, \dif x
- \mu (\Phi_j ) \right| < \frac{1}{i}.
\end{equation}
Since $\Phi_{0} = | \cdot |$ is included above we infer that the sequence $( \xi +\bala  u_i )$ is bounded in $\lebe^1 (\X,\V)$ and so,
by Banach-Alaoglu's compactness theorem and separability of $\HH$, each of its subsequences admit a further subsequence (not relabelled) so
\begin{equation}\label{closed2}
\int_{\X} \! \Phi (z+ \bala u_i ) \, \dif x \to \ell (\Phi ) \quad \forall \Phi \in \HH 
\end{equation}
for some $\ell \in \HH^\ast$ that at this stage of course can depend on the particular subsequence. But the density of $\{ \Phi_j \}_{j \in \N_0}$ and the
identification of $\HH^\ast$ with pairs of measures allow us by virtue of (\ref{closed1}) to conclude that $\ell = \mu$ and so that (\ref{closed2})
in fact holds true for the full sequence. Taken together with (\ref{closed1}) we have shown that $\mu \in Y$, and the proof is complete.
\end{proof}

\begin{lemma}\label{cvx}
The set $Y$ is convex.
\end{lemma}  

\begin{proof}
Let $\nu_0$, $\nu_1 \in Y$, $t \in (0,1)$ and $\nu_t = (1-t)\nu_0 + t\nu_1$. 
In view of Lemmas \ref{test} and \ref{closed} it suffices to show that $\nu_t \in Y$ when
$$
\nu_0 = \epy_{z+ \bala u_{0}}, \, \nu_1 = \epy_{z+ \bala u_{1}} \, \mbox{ and } \, t=\left( \tfrac{p}{q} \right)^{n},
$$
where $u_0$, $ u_1 \in \CC^{\infty}_{c}( \X, \U )$ and $p$, $q \in \N$ are coprime.
In fact, we shall show that $\nu_t = \epy_{z+\bala \phi}$ for some $\phi \in \CC^{\infty}_{c}( \X, \U )$.
	
In order to see this we subdivide each side of $\X$ into $q$ disjoint congruent half-open intervals, and consider the resulting
mesh of $q^n$ disjoint congruent cubes. We fix an arrangement of the cubes, say
$$
\X_i = x_i + q^{-1}\X , \quad i \in \{ 1, \, \dots \, , \, q^n \} ,
$$
and define
$$
\phi (x) = \sum_{i=1}^{p^n} q^{-l} u_{1} \bigl( q(x-x_{i}) \bigr) + \sum_{i=p^{n}+1}^{q^n} q^{-l} u_{0} \bigl( q(x-x_{i}) \bigr)
\quad (x \in \X )
$$
where we extend $ u_0$, $ u_1$ to $\Rn \setminus \X$ by $0$. Clearly, $\phi \in \CC^{\infty}_{c}( \X , \U)$ and for $\Phi \in \HH$ we
check that $\epy_{z+ \bala \phi}( \Phi ) = \nu_t ( \Phi )$, and therefore $\nu_t \in Y$ by Lemma \ref{test}.
\end{proof}

For a continuous integrand $f \colon \V \to \R$ of linear growth we define the upper recession integrand by
$$
f^{\infty}(z) = \limsup_{\stackrel{t \to \infty}{z^{\prime} \to z}} \frac{f(tz^{\prime})}{t} \quad (z \in \V )
$$
Hereby $f^{\infty} \colon \V \to \R$ is a positively $1$-homogeneous integrand of linear growth. When the wave cone for $\cala$
spans the space $\V$, then $\cala$-quasiconvex integrands $f$ of linear growth are Lipschitz continuous. In that case, the upper
recession integrand $f^{\infty}$ is also positively $1$-homogeneous, Lipschitz and $\cala$-quasiconvex. Furthermore in this situation
we also have the simpler formula
$$
f^{\infty}(z) = \limsup_{t \to \infty} \frac{f(tz)}{t} \quad (z \in \V ).
$$
\begin{lemma}\label{hb}
Let $\nu = (\nu^0 , \nu^\infty ) \in \Meas ( \V ) \times \mathcal{M}^{+}(\SV )$ and $z \in \V$. Then $\nu \in Y$ if, and only if,
$\overline{\nu}^{0}+\overline{\nu}^{\infty} = z$ and
$$
\int_{\V} \! f \, \dif \nu^0 + \int_{\SV} \! f^{\infty} \, \dif  \nu^{\infty} \geq f(z)
$$
holds for all $\cala$-quasiconvex integrands $f \colon \V \to \R$ of linear growth.
\end{lemma}

\begin{proof}
To prove the \textit{only if} part we fix $\nu \in Y$ and an $\cala$-quasiconvex integrand $f \colon \V \to \R$ of linear growth.
Select a sequence $(  u_j )$ in $\CC^{\infty}_{c}( \X ,\U)$ such that $\|  u_j \|_{\sobo^{l-1,1}} \to 0$ and
$$
\int_{\X} \! \Phi (z + \bala u_{j}) \, \dif x \to  \nu (\Phi )
$$
for each $\Phi \in \HH$. Thus for $\Phi \in \HH$ with $\Phi \geq f$ we get by $\cala$-quasiconvexity of $f$ at $z$ that
$$
\nu ( \Phi )  \geq \limsup_{j \to \infty}\int_{\X} \! f(z + \bala u_{j}) \, \dif x \geq f(z).
$$
To obtain the required bound we use the approximation Lemma \ref{approximationlemma} and the Lebesgue monotone convergence theorem.
	
We turn to the \textit{if} part of the statement. According to Lemmas \ref{closed} and \ref{cvx}, $Y$ is a weakly$\mbox{}^\ast$ closed and convex
subset of $\HH^\ast$, hence by the Hahn--Banach separation theorem we can write $Y = \bigcap H$,
where we take intersection over all weakly$\mbox{}^\ast$ closed half-spaces $H$ in $\HH^\ast$ that contain $Y$. Fix such a half-space $H$; it is
well-known that we can find $\Phi \in \HH$, $t \in \R$ such that
$$
H = \bigl\{ \ell \in \HH^\ast : \, \ell ( \Phi ) \geq t \bigr\} .
$$
Because $Y \subset H$ it follows in particular from Lemma \ref{test} that $\epy_{z + \bala u} \in H$ for all $ u \in \CC^{\infty}_{c}( \X ,\U)$.
Consequently, the inequality
$$
t \leq \epy_{z + \bala u}( \Phi ) = \int_{\X} \! \Phi (z+ \bala u ) \, \dif x
$$
holds for all $ u \in \CC^{\infty}_{c}( \X,\U )$. But then the relaxation formula \eqref{eq:relax_formula_YM} yields $t \leq \Phi^{\mathrm{qc}}(z)$
and it follows from Lemma~\ref{lem:finite_valued} that the envelope $\Phi^{\mathrm{qc}}$ is real-valued, $\cala$-quasiconvex and of linear growth.
Returning to our assumptions on $\nu$ with this information we deduce that
$$
\nu (\Phi ) \geq \nu ( \Phi^{\mathrm{qc}}) \geq \Phi^{\mathrm{qc}}(z ) \geq t,
$$
that is, $\nu \in H$.
\end{proof}

\subsection{Inhomogenization}\label{sec:inhom}

Let $\bigl( \phi_{t} \bigr)_{t>0}$ be a standard smooth mollifier on $\Rn$: $\phi_{t}(x) = t^{-n}\phi \bigl( x/t \bigr)$, where
$\phi \colon \Rn \to \R$ is a nonnegative, $\CC^\infty$ smooth and compactly supported function that integrates to $1$.
It is convenient to assume that the support of $\phi$ is the closed unit cube $\overline{\X} = [ -\tfrac{1}{2},\tfrac{1}{2}]^{n}$ and that $\phi$
is strictly positive on its interior $\X$. For later reference we put
\begin{equation}\label{mollifierbound}
M \equiv \max | \nabla \phi | .
\end{equation}
We shall also in this subsection use the $\ell^{n}_{\infty}$-metric on $\Rn$ and for
convenience of notation we denote it simply by $\| \cdot \|$, thus
$$
\| x \| \equiv \| x \|_{\ell^{n}_{\infty}} = \max \bigl\{ |x_{1}|, \, \dots \, , \, |x_{n}| \bigr\} \, \mbox{ for }
x = \bigl( x_{1}, \, \dots \, , \, x_{n} \bigr) \in \Rn .
$$
The following is the core novelty of this work:
\begin{lemma}\label{singgen}
Under the assumptions of Theorem \ref{main}, given $\varepsilon > 0$ we can find $t_{\varepsilon}>0$ and
$\varphi = \varphi_{t} \in \CC^{\infty}_{c}( \Omega , \U )$ with $\| \varphi \|_{\WW^{l-1,1}( \Omega , \U )} < \varepsilon$ so
\begin{equation}\label{saim}
\left| \int_{\Omega} \! \eta \Phi (0) \dd x + \int_{\Omega} \! \langle \Phi^{\infty},\nu^{\infty}_{x} \rangle \, \dd \lambda^{s} 
\int_{\Omega} \! \eta \Phi \bigl( \phi_{t} \ast (\overline{\nu}^{\infty}_{\cdot} \lambda^{s}) + \bala \varphi \bigr) \, \dd x \right| < \varepsilon
\end{equation}
holds for $t \in (0,t_{\varepsilon}]$ uniformly in $\eta \colon \overline{\Omega} \to \R$ and $\Phi \colon \V \to \R$ of class $\HH$ with
\begin{equation}\label{stest}
\| \eta \|_{\LIP} \leq 1 \,  \mbox{ and } \, \| T \Phi \|_{\LIP} \leq 1.
\end{equation}
\end{lemma}
In line with what we stated at the start of this subsection we use the $\ell^{n}_{\infty}$-metric in the first bound of (\ref{stest}):
$$
\| \eta \|_{\LIP} \equiv \sup_{\Omega} | \eta | + \sup_{x, y \in \Omega \, x \neq y} \frac{| \eta (x)-\eta (y)|}{\| x-y \|}.
$$
It is elementary to check that $\|D\Phi\|_{\infty}\leq 3\|DT\Phi\|_{\infty}$. We will also assume implicit a renormalization such that $\|T\Phi\|_\mathrm{LIP}\leq 1$ implies $\|D\Phi\|_\infty\leq 1$.
\begin{proof}
Let $\varepsilon \in (0,1)$. Apply Luzin's theorem to the $\lambda^s$ measurable map
$$
\overline{\Omega} \ni x \mapsto \bigl( \delta_{0}, \nu_{x}^{\infty} \bigr) \in \mathcal{M}^{+}_{1}( \V ) \times \mathcal{M}^{+}( \SV ) \hookrightarrow \bigl( \HH^{\ast} \bigr)^{+}
$$
to find a compact subset $C^{s}=C^{s}( \varepsilon ) \subset \overline{\Omega}$ so that
$\lambda^{s} ( \overline{\Omega} \setminus C^s ) < \varepsilon \lambda^{s}( \overline{\Omega})$ and  the map
$$
C^{s} \ni x \mapsto \bigl( \delta_{0}, \nu_{x}^{\infty} \bigr) \in \mathcal{M}^{+}_{1}( \V ) \times \mathcal{M}^{+}( \SV )
$$
is (uniformly) continuous. The latter can of course be expressed quantitatively in the sense that we can find a modulus of continuity
$\omega^{s} = \omega^{s}_{\varepsilon} \colon [ 0,\infty ) \to [ 0, \infty )$ such that for all $x$, $y \in C^s$ 
the inequality
\begin{equation}\label{unibounds}
\bigl\| \bigl( \delta_{0}, \nu_{x}^{\infty} \bigr)-\bigl( \delta_{0}, \nu_{y}^{\infty} \bigr) \bigr\|_{\mathrm{K}}
\leq \omega^{s} \bigl( \|x-y\| \bigr) 
\end{equation}
holds. Because $\lambda^s$ and $\lebn$ are mutually singular we can assume that $\lebn (C^{s})=0$ and because $\lambda^{s} (\partial \Omega ) =0$
we can also assume that
\begin{equation}\label{distfromctobdry}
\Delta = \Delta_{\varepsilon} \equiv \inf \bigl\{ \| x-y \| : \, x \in C^{s} , \, y \in \partial \Omega \bigr\} \in (0,1).
\end{equation}    
Let $d$, $m \in \N$ be two integers whose precise values will be specified in the course of the proof.
Put $t = 2^{-d}$ and for convenience of notation write $\phi \equiv \phi_{t}$ so that we in particular have
\begin{equation}\label{supportphi}
\phi (x) \left\{
\begin{array}{ll}
  > 0 & \mbox{ if } \| x \| < t\\
  =0 & \mbox{ if } \| x \| \geq t.
\end{array}
\right.
\end{equation}  
Our first condition on $d \in \N$ is that it is so large that 
\begin{equation}\label{condition1}
2t \leq \Delta \mbox{ that is } d \geq \log_{2} \left( \tfrac{2}{\Delta} \right) .
\end{equation}
The collection of $(d+m)$-th generation dyadic cubes $Q \in \Dy_{d+m}$ in $\Rn$ with
$\mathrm{dist}(Q, \partial \Omega ) \geq t$ is denoted $\mathcal{F}$. For each $Q \in \mathcal{F}$ we define
\begin{equation}\label{rfunctional}
r_Q \equiv \dashint_{Q} \! \phi \ast (\lambda^{s} \restrict C^{s}) \, \dd x.
\end{equation}
In view of the choice of mollifier $\phi$ we have for cubes $Q$ with $r_{Q} > 0$ that
$$
\mathrm{dist}(Q,C^{s}) \equiv \inf \bigl\{ \| x-y \| : \, x \in Q, \, y \in C^{s} \bigr\} < t,
$$
and may therefore select $x_{Q} \in C^{s}$ so
$$
\mathrm{dist}(Q,x_{Q}) \equiv \inf \bigl\{ \| x-x_{Q} \| : \, x \in Q \bigr\} < t.
$$
Summarizing, we denote $\mathcal{F}^{s} \equiv \{ Q \in \mathcal{F}: \, r_{Q} > 0 \}$ and have that
\begin{equation}\label{xq}
\forall \, Q \in \mathcal{F}^{s} \, \exists \, x_{Q} \in C^{s} \, \mbox{ such that } \, \sup_{x \in Q} \| x-x_{Q} \| < 2t.
\end{equation}  
We continue the selection process for these cubes and fix $Q \in \mathcal{F}^s$. Then for integrands $f \in \HH$ that
are $\cala$-quasiconvex we get by Lemmas \ref{qcrc}, \ref{rcoflineargrowth}, \ref{rankonetheorem} and \ref{1homo},
\begin{eqnarray*}
f \bigl( r_{Q}\overline{\nu}_{x_{Q}}^{\infty} \bigr) &\leq& f \bigl( 0 \bigr) + r_{Q}f^{\infty} \bigl( \overline{\nu}_{x_{Q}}^{\infty} \bigr)\\
                                                     &\leq&    f \bigl( 0 \bigr) + r_{Q} \int_{\SV} \! f^{\infty} \, \dd \nu_{x_{Q}}^{\infty}.
\end{eqnarray*}       
In connection with the Hahn--Banach Lemma \ref{hb} this bound allows us to select $\varphi^{Q} \in \CC^{\infty}_{c}( Q ,\U )$ satisfying
$\| \varphi^{Q} \|_{\WW^{l-1,1}(Q, \U )} < \varepsilon \lambda^{s}(Q)$ and
\begin{equation}\label{sgeneration1}
\bigl\| \bigl( \delta_{0},\nu_{x_{Q}}^{\infty}r_{Q} \bigr) - \epy_{r_{Q}\overline{\nu}_{x_{Q}}^{\infty}+\bala \varphi^{Q}} \bigr\|_{\mathrm{K}}
< \varepsilon .
\end{equation}
We recall that this amounts to (where we realize the elementary Young measure on $Q$ instead of $\X$),
$$
\sup_{\| T\Phi \|_{\LIP} \leq 1} \left| \Phi (0)+\int_{\SV} \! \Phi^{\infty} \, \dd \nu^{\infty}_{x_{Q}} r_{Q}-\dashint_{Q} \!
\Phi \bigl( \overline{\nu}^{\infty}_{x_{Q}}r_{Q} + \bala \varphi^{Q} \bigr) \, \dd x \right| < \varepsilon .
$$
If therefore we define
$$
\varphi \equiv \sum_{Q \in \mathcal{F}^s} \varphi^{Q},
$$
then clearly
\begin{equation}\label{sgeneration2}
\varphi \in \CC^{\infty}_{c}( \Omega , \U ) , \quad \| \varphi \|_{\WW^{l-1,1}( \Omega , \U )} < \varepsilon \lambda^{s}( \Omega )
\end{equation}
and we have (\ref{sgeneration1}) for each $Q \in \mathcal{F}^s$. The sought-after map is now
\begin{equation}\label{smap}
\xi^{s} \equiv \phi \ast \bigl( \overline{\nu}^{\infty}_{\cdot}\lambda^{s} \bigr) + \bala \varphi ,
\end{equation}  
where we use the conventions that $\overline{\nu}^{\infty}_{\cdot}\lambda^{s} \equiv 0$ and $\varphi \equiv 0$ off $\Omega$ so that
$\xi^{s} \in \CC^{\infty}_{c}( \Rn , \V )$. In order to check that $\xi^s$ has the required properties, in addition to (\ref{sgeneration2}), we
fix $\eta \colon \overline{\Omega} \to \R$ and $\Phi \colon \V \to \R$ of class $\HH$ satisfying (\ref{stest}).
We start by writing
$$
\int_{\Omega} \! \eta \langle \Phi^{\infty} , \phi \ast \bigl( \nu^{\infty}_{\cdot}\lambda^{s} \bigr) \rangle \, \dd x
= \int_{\Omega} \! \eta \langle \Phi^{\infty} , \phi \ast \bigl( \nu^{\infty}_{\cdot}\lambda^{s} \restrict C^{s} \bigr) \rangle \, \dd x
+\mathcal{E}_{1},                                                                                                                          
$$
where
$$
| \mathcal{E}_{1}| \leq \int_{\Omega} \! \phi \ast \bigl( \lambda^{s} \restrict \Omega \setminus C^{s} \bigr) \, \dd x \leq \varepsilon \lambda^{s}( \Omega ).
$$
We emphasize that the above integrals should be understood as
$$
\int_{\Omega} \! \eta \langle \Phi^{\infty} , \phi \ast \bigl( \nu^{\infty}_{\cdot}\lambda^{s} \bigr) \rangle \, \dd x \equiv
\int_{\Omega} \! \eta (x) \int_{\Omega} \! \phi (x-y) \int_{\SV} \! \Phi^{\infty} \, \dd \nu^{\infty}_{y} \, \dd \lambda^{s}(y) \, \dd x,
$$
and likewise for the integral on the right-hand side, except that the $y$-integration is over the set $C^s$ instead of $\Omega$.
Since for each $Q \in \mathcal{F}$ with $r_{Q}=0$ we have
$$
\int_{Q} \! \eta \langle \Phi^{\infty} , \phi \ast \bigl( \nu^{\infty}_{\cdot}\lambda^{s} \restrict C^{s} \bigr) \rangle \, \dd x = 0,
$$
and $\mathrm{dist}( \partial \Omega , \bigcup \mathcal{F}) \leq t$ we get
\begin{eqnarray*}
\int_{\Omega} \! \eta \langle \Phi^{\infty} , \phi \ast \bigl( \nu^{\infty}_{\cdot}\lambda^{s} \restrict C^{s} \bigr) \rangle \, \dd x
&=& \sum_{Q \in \mathcal{F}^s} \! \int_{Q} \! \eta \langle \Phi^{\infty} , \phi \ast \bigl( \nu^{\infty}_{\cdot}\lambda^{s} \restrict C^{s} \bigr) \rangle \, \dd x+ \mathcal{E}_{2}\\
&=& \sum_{Q \in \mathcal{F}^s} \! \left( \int_{Q} \! \eta \, \dd x \langle \Phi^{\infty},\nu^{\infty}_{x_{Q}} \rangle r_{Q} + \mathcal{E}_{3}^{Q} \right) + \mathcal{E}_{2},
\end{eqnarray*}
where
$$
\bigl| \mathcal{E}_{2} \bigr| \leq \lambda^{s} \bigl( C^{s} \cap ( \partial \Omega )_{2t} \bigr) \, \mbox{ with }
\, ( \partial \Omega )_{2t} \equiv B_{2t}(0)+\partial \Omega ,
$$
so that $\mathcal{E}_{2}=0$ according to (\ref{distfromctobdry}) and (\ref{condition1}).
The local error terms $\mathcal{E}_{3}^Q$ are estimated as follows. First,
\begin{eqnarray*}
\bigl| \mathcal{E}_{3}^{Q} \bigr| &\leq& \left| \int_{Q} \! \bigl( \eta - \dashint_{Q} \! \eta \bigr) \langle \Phi^{\infty},
                                         \phi \ast \bigl( \nu^{\infty}_{\cdot}\lambda^{s} \restrict C^{s} \bigr)\rangle \, \dd x\right|\\
                                  && + \left| \int_{Q} \! \eta \left( \dashint_{Q} \! \langle \Phi^{\infty},\phi \ast \bigl( \nu^{\infty}_{\cdot}
                                     \lambda^{s} \restrict C^s \bigr) \rangle \, \dd x - \langle \Phi^{\infty},\nu^{\infty}_{x_{Q}} \rangle r_{Q} \right) \right|\\
                                  &\leq& \| \eta \|_{\LIP} \lebn (Q)^{\frac{1}{n}} \max_{\SV} | \Phi^{\infty}| \int_{Q} \! \phi \ast \lambda^{s} \, \dd x\\
                                  &&  + \| \eta \|_{\LIP}\int_{Q} \! \int_{C^s} \! \phi(x-y)\bigl\langle \Phi^{\infty},\nu^{\infty}_{y}
                                     -\nu^{\infty}_{x_{Q}}\bigr\rangle \, \dd \lambda^{s}(y) \, \dd x
\end{eqnarray*}
and invoking (\ref{xq}), (\ref{sgeneration1}) and (\ref{stest}) we continue with
\begin{eqnarray*}
\bigl| \mathcal{E}_{3}^{Q} \bigr|  &\leq& t \int_{Q} \!  \phi \ast \lambda^{s} \, \dd x \\
                                   && + \int_{Q} \! \int_{C^{s}} \! \phi (x-y) \omega^{s} \bigl( \| y-x_{Q} \| \bigr) \, \dd \lambda^{s}(y) \, \dd x\\
                                  &\leq& \bigl( t + \omega^{s} (3t) \bigr) \int_{Q} \! \phi \ast \lambda^{s} \, \dd x,
\end{eqnarray*}
where the last inequality follows from the triangle inequality and (\ref{xq}). Next, for each $Q \in \mathcal{F}^s$ we get from
(\ref{sgeneration1}) that
$$
\langle \Phi^{\infty},\nu^{\infty}_{x_{Q}} \rangle r_{Q} = \Phi (0) + \dashint_{Q} \! \Phi (r_{Q}\overline{\nu}^{\infty}_{x_{Q}}+\bala \varphi^{Q})
\, \dd x + \mathcal{E}^{Q}_{4},
$$
where $| \mathcal{E}_{4}^{Q}| \leq \varepsilon$. Using the triangle inequality and (\ref{sgeneration1}) again we estimate
\begin{eqnarray*}
\dashint_{Q} \! \bigl| \Phi \bigl( r_{Q}\overline{\nu}^{\infty}_{x_{Q}}+\bala \varphi^{Q} \bigr) - \Phi (0) \bigr| \, \dd x &\leq&
\| T \Phi \|_{\LIP}\dashint_{Q} \! \bigl| r_{Q}\overline{\nu}^{\infty}_{x_{Q}}+\bala \varphi^{Q} \bigr| \, \dd x\\
                                                                                                                            &\leq&
\| \epy_{r_{Q}\overline{\nu}^{\infty}_{x_{Q}}+\bala \varphi^{Q}} \|_{\mathrm{K}}\\
&\leq& \| \bigl( \delta_{0},\nu^{\infty}_{x_{Q}}r_{Q} \bigr) \|_{\mathrm{K}} + \varepsilon\\  
&\leq& 1+r_{Q}+\varepsilon
\end{eqnarray*}
and consequently,
$$
\int_{Q} \! \eta \, \dd x \dashint_{Q} \! \Phi (r_{Q}\overline{\nu}^{\infty}_{x_{Q}}+\bala \varphi^{Q}) \, \dd x =
\int_{Q} \! \eta \Phi (r_{Q}\overline{\nu}^{\infty}_{x_{Q}}+ \bala \varphi^{Q}) \, \dd x + \mathcal{E}^{Q}_{5},
$$
where (recall $\varepsilon < 1$)
\begin{eqnarray*}
| \mathcal{E}^{Q}_{5}| &\leq& \sup_{Q} \left| \eta - \dashint_{Q} \! \eta \right| \bigl( 1 + r_{Q} + \varepsilon \bigr) \lebn (Q)\\
&\leq&\lebn (Q)^{\frac{1}{n}} \int_{Q} \! \bigl( 2 + \phi \ast \lambda^{s} \bigr) \, \dd x.
\end{eqnarray*}
Finally we turn to
$$
\int_{Q} \! \eta \Phi (r_{Q}\overline{\nu}^{\infty}_{x_{Q}}+ \bala \varphi^{Q}) \, \dd x = \int_{Q} \! \eta \Phi \bigl( \phi \ast (\overline{\nu}^{\infty}_{\cdot}
\lambda^{s}) + \bala \varphi^{Q} \bigr) \, \dd x + \mathcal{E}^{Q}_{6}.
$$
To estimate the local error term $\mathcal{E}^{Q}_{6}$, $Q \in \mathcal{F}^s$, we start with the bound
$$
\left| \int_{Q} \! \eta \biggl( \Phi \bigl( \phi \ast (\overline{\nu}^{\infty}_{\cdot}
  \lambda^{s}) + \bala \varphi^{Q} \bigr)-   \Phi \bigl( \phi \ast (\overline{\nu}^{\infty}_{\cdot} \lambda^{s} \restrict C^{s}) +\bala \varphi^{Q} \bigr)
  \biggr) \, \dd x \right| \leq
\int_{Q} \phi \ast \bigl( \lambda^{s} \restrict \Omega \setminus C^{s} \bigr) \, \dd x
$$
that is an easy consequence of (\ref{stest}). Another application of (\ref{stest}) yields
$$
\left| \int_{Q} \! \eta  \biggl( \Phi \bigl( r_{Q}\overline{\nu}^{\infty}_{x_{Q}} + \bala \varphi^{Q} \bigr) -
\Phi \bigl( \phi \ast (\overline{\nu}^{\infty}_{\cdot}\lambda^{s} \restrict C^{s}) + \bala \varphi^{Q} \bigr) \biggr) \, \dd x \right|
\leq \int_{Q} \! \bigl| r_{Q}\overline{\nu}^{\infty}_{x_{Q}}-\phi \ast (\overline{\nu}^{\infty}_{\cdot}
\lambda^{s} \restrict C^{s})\bigr| \, \dd x .
$$
We estimate the last term as follows:
\begin{eqnarray*}
\int_{Q} \! \bigl| r_{Q}\overline{\nu}^{\infty}_{x_{Q}}-\phi \ast (\overline{\nu}^{\infty}_{\cdot}
  \lambda^{s} \restrict C^{s})\bigr| \, \dd x &\leq& \left| \int_{Q} \! \phi \ast \bigl( ( \overline{\nu}^{\infty}_{x_{Q}}-\overline{\nu}^{\infty}_{\cdot})
                                                     \lambda^{s} \restrict C^{s} \bigr) \, \dd x \right|\\
                                              && +\int_{Q} \left| \dashint_{Q} \! \phi \ast (\overline{\nu}^{\infty}_{\cdot}\lambda^{s} \restrict C^{s}) \, \dd x^{\prime}-
                                                 \phi \ast (\overline{\nu}^{\infty}_{\cdot}\lambda^{s} \restrict C^{s}) \right| \, \dd x\\
                                              &\leq& \int_{Q} \! \int_{C^{s}} \! \phi (x-y) \bigl| \overline{\nu}^{\infty}_{x_{Q}}-\overline{\nu}^{\infty}_{y} \bigr|
                                                     \, \dd \lambda^{s} \, \dd x\\
  && + \int_{Q} \left| \dashint_{Q} \! \phi \ast (\overline{\nu}^{\infty}_{\cdot}\lambda^{s} \restrict C^{s}) \, \dd x^{\prime}-
     \phi \ast (\overline{\nu}^{\infty}_{\cdot}\lambda^{s} \restrict C^{s}) \right| \, \dd x\\
  &\stackrel{(\ref{sgeneration1})}{\leq}& \omega^{s}( 3t ) \int_{Q} \! \phi \ast \lambda^{s} \, \dd x + \mathcal{E}^{Q}_{7}
\end{eqnarray*}
where the local error terms $\mathcal{E}^{Q}_{7}$ are the mean oscilations on $Q \in \mathcal{F}^s$ times $\lebn (Q)$. 
We write explicitly
\begin{eqnarray*}
\mathcal{E}_7^Q&\leq &\int_Q\dashint_Q \int_{C^s}\left|\phi(x^\prime -y)-\phi(x-y)\right||\overline\nu_y^\infty|\dif\lambda^s(y)\dif x^\prime\dif x\\
&\leq& \int_Q\dashint_Q\int_{Q+t\X}\|x^\prime-x\| \int_0^1|D\phi (x-y+\tau(x^\prime-x))|\dif \tau\dif\lambda^s(y)\dif x^\prime \dif x
\end{eqnarray*}

We now recall that $\phi=\phi_t$, so that we have the pointwise bounds $|D\phi|\leq t^{-1}M(\ONE_{\X})_t$, which implies for $x,\,x^\prime\in Q$, $y\in Q+t\X$, and $\tau\in[0,1]$ that
$$
|D\phi (x-y+\tau(x^\prime-x))|\leq t^{-1}M(\ONE_{2\X})_t(x-y),
$$
so that
$$
\mathcal{E}_7^Q\leq M \frac{\lebn (Q)^{\frac{1}{n}}}{t} \int_{Q} \! (\ONE_{2\X})_t \ast \lambda^{s} \, \dd x
$$
We now have all the necessary bounds and can start to backtrack through the estimates to conclude the proof of Lemma \ref{singgen}.
First we recall that we have $t=2^{-d}$ with $d \in \N$ satisfying (\ref{condition1}) and that the considered dyadic cubes are of higher
generation $Q \in \Dy_{d+m}$ (so $\lebn (Q)=2^{-n(d+m)}$). With $\xi^s$ defined in (\ref{smap}) we get by combination of the above
$$
\int_{\Omega} \! \eta \langle \Phi^{\infty},\phi \ast \bigl( \nu^{\infty}_{\cdot}\lambda^{s} \bigr) \rangle \, \dd x =
\int_{\bigcup \mathcal{F}^s} \! \eta \Phi ( \xi^{s}) \, \dd x + \mathcal{E}
$$
where
\begin{eqnarray*}
  | \mathcal{E}| &\leq& \varepsilon \lambda^{s}( \Omega ) + \bigl( t + \omega^{s}(3t) \bigr) \int_{\bigcup \mathcal{F}^s} \! \phi \ast \lambda^{s} \, \dd x\\
                 && +2^{-d-m}\int_{\bigcup \mathcal{F}^s} \! \bigl( 2+ \phi \ast \lambda^{s} \bigr) \, \dd x + \int_{\bigcup \mathcal{F}^s} \! \phi \ast \bigl(
                    \lambda^{s} \restrict \Omega \setminus C^{s} \bigr) \, \dd x\\
  && + \omega^{s}(3t)\int_{\bigcup \mathcal{F}^s} \! \phi \ast \lambda^{s} \, \dd x + M2^{-m}\int_{\bigcup \mathcal{F}^s} \! (\ONE_{2\X})_{t} \ast \lambda^{s} \, \dd x\\
                 &\leq& \bigl( 2\varepsilon +2^{-d}+2\omega^{s} \bigl(3 \cdot 2^{-d} \bigr) + 2^{-d-m}+c_{n}M2^{-m} \bigr) \lambda^{s}( \Omega )\\
  && + 2^{1-d-m}\lebn ( \Omega ).
\end{eqnarray*}
To conclude we add
$$
\int_{\Omega \setminus \bigcup \mathcal{F}^{s}} \! \eta \Phi ( \xi^{s}) \, \dd x = \int_{\Omega \setminus \bigcup \mathcal{F}^{s}} \! \eta \, \dd x \Phi (0)
$$
to both sides, whereby we get
\begin{equation}\label{sslut}
\int_{\Omega}  \! \eta \biggl( \Phi (0)+\langle \Phi^{\infty},\phi \ast \bigl( \nu^{\infty}_{\cdot}\lambda^{s} \bigr) \rangle \biggr) \, \dd x =
\int_{\Omega} \! \eta \Phi ( \xi^{s}) \, \dd x + \mathcal{E}+\int_{\bigcup \mathcal{F}^{s}} \! \eta \, \dd x \Phi (0).
\end{equation}
Here we have that $\bigcup \mathcal{F}^{s} \subset C^{s} + \overline{B}_{2t}(0)$ and since $\lebn (C^{s})=0$ we may clearly find $d_{\varepsilon}$, $m_{\varepsilon} \in \N$
depending only on $\varepsilon > 0$ so that for $d \geq d_{\varepsilon}$, $m \geq m_{\varepsilon}$ we have that
$$
| \mathcal{E}| + \left| \int_{\bigcup \mathcal{F}^{s}} \! \eta \, \dd x \Phi (0) \right| \leq 3\varepsilon \bigl( \lebn + \lambda^{s} \bigr)( \Omega ).
$$
This completes the proof since the left-hand side of (\ref{sslut}) clearly tends to
$$
\int_{\Omega} \! \eta \, \dd x \Phi (0)+ \int_{\Omega} \! \eta \langle \Phi^{\infty},\nu^{\infty}_{x } \rangle \, \dd \lambda^{s}(x)
$$
uniformly in $\eta$, $\Phi$ satisfying (\ref{stest}) as $d \to \infty$.
\end{proof}

\begin{lemma}\label{reggen}
	Under the assumptions of Theorem \ref{main}, given $\varepsilon > 0$ we can find $t_{\varepsilon}>0$ and
	$\psi = \psi_{t} \in \CC^{\infty}_{c}( \Omega , \U )$ with $\| \psi \|_{\WW^{l-1,1}( \Omega , \U )} < \varepsilon$ so
	\begin{equation}
	\left| \int_{\Omega} \! \eta \left(\langle\Phi,\nu_x\rangle  + \! \lambda^a(x)\langle \Phi^{\infty},\nu^{\infty}_{x} \rangle  \right)\dd x-
	\int_{\Omega} \! \eta \Phi \bigl(\phi_t*(\overline{\nu}_{\cdot}+\lambda^a\overline{\nu}^\infty_\cdot ) + \bala \psi \bigr) \, \dd x \right| < \varepsilon
	\end{equation}
	holds for $t \in (0,t_{\varepsilon}]$ uniformly in $\eta \colon \overline{\Omega} \to \R$ and $\Phi \colon \V \to \R$ of class $\HH$ with
	\begin{equation}
	\| \eta \|_{\LIP} \leq 1 \,  \mbox{ and } \, \| T \Phi \|_{\LIP} \leq 1.
	\end{equation}
\end{lemma}

\begin{proof}
	Let $\varepsilon \in (0,1)$. Apply Luzin's theorem to the $\mathscr{L}^n$ measurable map
	\begin{eqnarray}\label{eq:map}
	{\Omega} \ni x \mapsto \bigl( \nu_{x}, \lambda^a(x)\nu_{x}^{\infty} \bigr) \in \mathcal{M}^{+}_{1}( \V ) \times \mathcal{M}^{+}( \SV ) \equiv \bigl( \HH^{\ast} \bigr)^{+}
	\end{eqnarray}
	to find a compact subset $C^{a}=C^{a}( \varepsilon ) \subset {\Omega}$ so that
	$|{\Omega} \setminus C^a| < \varepsilon |{\Omega}|$ and we can find a modulus of continuity
	$\omega^{a} = \omega^{a}_{\varepsilon} \colon [ 0,\infty ) \to [ 0, \infty )$ such that for all $x$, $y \in C^a$ 
	the inequality
	\begin{equation}\label{eq:Kantorovic_bound}
	\bigl\| \bigl( \nu_{x}, \lambda^a(x)\nu_{x}^{\infty} \bigr)-\bigl( \nu_{y}, \lambda^a(y)\nu_{y}^{\infty} \bigr) \bigr\|_{\mathrm{K}}
	\leq \omega^{a} \bigl( \|x-y\| \bigr) 
	\end{equation}
	holds. We can moreover assume by Luzin's and Tietze's theorems that we there exists a function $g\in\hold(\overline{\Omega})$ such that $g=\lambda^a$ in $C^a$ and $g$ also has modulus of continuity $\omega^a$ (in $\overline{\Omega}$). 
	 We consider a dyadic grid on $\R^n$ with step-size $t=2^{-d}$ and write 
	 $$
	 \mathcal{F}^a=\{Q\in\mathscr D_d\colon \mathrm{dist}(Q,\partial\Omega)>t,\,Q\cap C^a\neq\emptyset\},
	 $$
	where we recall that we work with the $\ell^\infty$ norm on $\R^n$. It is clear that for $d$ large enough we have that $|\bigcup\mathcal{F}^a|>(1-2\varepsilon)|\Omega|$. We let $x_Q\in Q\cap C^a$ be otherwise arbitrary. By Lemma~\ref{hb}, we obtain the existence of $\psi^{Q} \in \CC^{\infty}_{c}( Q ,\U )$ with
$\| \psi^{Q} \|_{\WW^{l-1,1}(Q, \U )} < \varepsilon |Q|$ and
\begin{equation}\label{eq:HB_AC}
\bigl\| \bigl( \nu_{x_Q},\lambda^a(x_Q)\nu_{x_{Q}}^{\infty}\bigr) - \epy_{\overline{\nu}_{x_Q}+\lambda^a(x_Q)\overline{\nu}_{x_{Q}}^{\infty}+\bala \psi^{Q}} \bigr\|_{\mathrm{K}}
< \varepsilon,
\end{equation}
	
	We then estimate:
	$$
	\int_{\Omega} \! \eta \left(\langle\Phi,\nu_x\rangle  + \! \lambda^a(x)\langle \Phi^{\infty},\nu^{\infty}_{x} \rangle  \right)\dd x=
\sum_{Q\in\mathcal{F}^a}	\int_{Q} \! \eta \left(\langle\Phi,\nu_x\rangle  + \! \lambda^a(x)\langle \Phi^{\infty},\nu^{\infty}_{x} \rangle  \right)\dd x+\mathcal{E}_1
	$$
	where
	$$
	|\mathcal{E}_1|\leq \int_{(\bigcup\mathcal{F}^a)^c}\int_{\mathbb{V}}1+|z|\dif\nu_x(z)+\lambda^a\dif x,
	$$
	which is arbitrarily small for large $d$ by the dominated convergence theorem and the moment condition. Next,
	$$
	\sum_{Q\in\mathcal{F}^a}	\int_{Q} \! \eta \left(\langle\Phi,\nu_x\rangle  + \! \lambda^a(x)\langle \Phi^{\infty},\nu^{\infty}_{x} \rangle  \right)\dd x=
	\sum_{Q\in\mathcal{F}^a}	\dashint_Q\eta\int_{Q} \!  \left(\langle\Phi,\nu_x\rangle  + \! \lambda^a(x)\langle \Phi^{\infty},\nu^{\infty}_{x} \rangle  \right)\dd x+\mathcal{E}_2,
	$$
	where
$$
	|\mathcal{E}_2|\leq \sum_{Q\in\mathcal{F}^a}	\int_{Q} \! \mathscr{L}^{n}(Q)^\frac{1}{n}\|\eta\|_{\mathrm{LIP}} \left|\langle\Phi,\nu_x\rangle  + \! \lambda^a(x)\langle \Phi^{\infty},\nu^{\infty}_{x} \rangle  \right|\dd x\leq t	\int_{\Omega} \! \int_\mathbb{V}1+|z|\dif\nu_x(z)+\lambda^a\dd x.
$$
We next look at
$$
\sum_{Q\in\mathcal{F}^a}	\dashint_Q\eta\int_{Q} \!  \left(\langle\Phi,\nu_x\rangle  + \! \lambda^a(x)\langle \Phi^{\infty},\nu^{\infty}_{x} \rangle  \right)\dd x=
\sum_{Q\in\mathcal{F}^a}	 \left(\langle\Phi,\nu_{x_Q}\rangle  + \! \lambda^a(x_Q)\langle \Phi^{\infty},\nu^{\infty}_{x_Q} \rangle  \right)\int_Q\eta+\mathcal{E}_3,
$$
where
\begin{eqnarray*}
|\mathcal{E}_3|&\leq& \sum_{Q\in\mathcal{F}^a}	\left|\left(\int_{Q\cap C^a}+\int_{Q\setminus C^a}\right) \!  \langle\Phi,\nu_x\rangle  + \! \lambda^a(x)\langle \Phi^{\infty},\nu^{\infty}_{x} \rangle-\langle\Phi,\nu_{x_Q}\rangle  - \! \lambda^a(x_Q)\langle \Phi^{\infty},\nu^{\infty}_{x_Q} \rangle  \dd x\right|\\
&\leq& \omega^a(t)|\Omega|+\int_{\Omega\setminus C^a} \int_\mathbb{V}1+|z|\dif\nu_x(z)+\lambda^a\dd x+\sum_{Q\in\mathcal{F}^a}|Q\setminus C^a|\lambda^a(x_Q).
\end{eqnarray*}
Here the first two terms are known to be small, as above, whereas for the third one we argue as follows:
\begin{eqnarray*}
\sum_{Q\in\mathcal{F}^a}|Q\setminus C^a|\lambda^a(x_Q)&=&\sum_{Q\in\mathcal{F}^a}|Q\setminus C^a|g(x_Q)= \sum_{Q\in\mathcal{F}^a}\int_{Q\cap C^a}\left[g(x_Q)-g(x)\right]+g(x)\dif x\\
&\leq&\omega^a(t)|\Omega\setminus C^a|+\int_{\Omega\setminus C^a}|g|\dif x,
\end{eqnarray*}
which can be made small. By \eqref{eq:HB_AC} we have that
$$
\langle\Phi,\nu_{x_Q}\rangle  + \! \lambda^a(x_Q)\langle \Phi^{\infty},\nu^{\infty}_{x_Q} \rangle =\dashint_{Q}\Phi\left(\overline{\nu}_{x_{Q}}+\lambda^a(x_Q)\overline{\nu}^{\infty}_{x_{Q}}+\bala \psi^Q\right)\dif x+\mathcal{E}_4^Q,
$$
where $|\mathcal{E}_4^Q|\leq \varepsilon$, so that
$$
\sum_{Q\in\mathcal{F}^a}	 \left(\langle\Phi,\nu_{x_Q}\rangle  + \! \lambda^a(x_Q)\langle \Phi^{\infty},\nu^{\infty}_{x_Q} \rangle  \right)\int_Q\eta
=
\sum_{Q\in\mathcal{F}^a}\int_{Q}\eta\Phi\left(\overline{\nu}_{x_{Q}}+\lambda^a(x_Q)\overline{\nu}^{\infty}_{x_{Q}}+\bala \psi^Q\right)\dif x+\mathcal{E}_5,
$$
where now
\begin{eqnarray*}
|\mathcal{E}_5|&\leq& \sum_{Q\in\mathcal{F}^a}\int_Q \left|\eta-\dashint_Q\eta\right|\left|\Phi\left(\overline{\nu}_{x_{Q}}+\lambda^a(x_Q)\overline{\nu}^{\infty}_{x_{Q}}+\bala \psi^Q\right)\right|\dif x\\
&\leq &\sum_{Q\in\mathcal{F}^a}t\left(\bigl\| \bigl( \nu_{x_Q},\lambda^a(x_Q)\nu_{x_{Q}}^{\infty} \bigr)  \bigr\|_{\mathrm{K}}|Q|
+ \varepsilon\right)\leq t(\varepsilon+|\Omega|S),
\end{eqnarray*}
where $S$ is the supremum over the compact set $C^a$ of the map in \eqref{eq:map}, which is obviously continuous there. We abbreviate $v^a=\overline{\nu}_\cdot+\lambda^a\overline{\nu}^\infty_\cdot=\langle \mathrm{id},\nu_\cdot\rangle+\lambda^a\langle\mathrm{id},\nu_\cdot^\infty\rangle$, the absolutely continuous part of the barycentre of $\nu$. We have
$$
\sum_{Q\in\mathcal{F}^a}\int_{Q} \eta\Phi\left(v^a(x_Q)+\bala\psi^Q\right)\dif x=\sum_{Q\in\mathcal{F}^a}\int_{Q}\eta\Phi\left(v^a\restrict C^a+\bala\psi^Q\right)\dif x+\mathcal{E}_6^Q,
$$
where
\begin{eqnarray*}
|\mathcal{E}_6^Q|&\leq &\int_{Q}|v^a(x_Q)-v^a\restrict C^a|\dif x\leq|Q\setminus C^a|S+\omega^a(t)|Q\cap C^a|
\end{eqnarray*}
where in the last inequality we used \eqref{eq:Kantorovic_bound} for the integrands $\Phi_i(z)=z_i$ (components of the identity).
In particular, $\sum_{\mathcal{F}^a}|\mathcal{E}_6^Q|\leq |\Omega\setminus C^a|S+ \omega^a(t)|\Omega|$.

We finally estimate
$$
\sum_{Q\in\mathcal{F}^a}\int_{Q}\eta\Phi\left(v^a\restrict C^a+\bala\psi^Q\right)\dif x=\int_{\Omega} \eta\Phi\left(v^a+\bala\psi\right)\dif x+\mathcal{E}_7,
$$
where $\psi=\sum_{\mathcal{F}^a}\psi^Q$, so that $$|\mathcal{E}_7|\leq \int_{\Omega\setminus\bigcup\mathcal{F}^a}|v^a|\dif x,$$
which can be made small since $v^a$ is integrable, so that moreover $\phi_t*v^a\rightarrow v^a$ in $\lebe^1(\Omega,\mathbb{V})$.
Collecting the above estimates we obtain 
$$
\left|\int_{\Omega}\eta\left(\langle \Phi,\nu_\cdot\rangle+\lambda^a \langle \Phi^\infty,\nu^\infty_\cdot\rangle\right)\dif x
  -\int_{\Omega}\eta\Phi(\phi_t*v^a+\bala \psi_\varepsilon)\dif x\right|<\delta(t,\varepsilon,\nu),
$$
which can be made arbitrarily small as $t,\,\varepsilon\downarrow 0$. The proof is complete.
\end{proof}
The proof of Theorem~\ref{main} then follows by Lemmas~\ref{singgen},~\ref{reggen},~\ref{lem:identify} and the elementary Lemma~\ref{lem:genseq} below.
\begin{lemma}\label{lem:genseq}
Suppose that the sequences $(v_j) ,\,(\tilde{v}_j) \subset\hold^\infty(\overline\Omega,\V)$ generate the Young measures
$\nu_1=((\nu_x)_{x\in\Omega},\lambda_1,(\nu_x^\infty)_{x\in\Omega}),\,\nu_2=((\delta_0)_{x\in\Omega},\lambda_2,(\nu_x^\infty)_{x\in\overline\Omega})$ respectively.
Suppose also that $\lambda_1\ll \lebn\perp\lambda_2$. Then $(v_j+\tilde v_j)$ generates $\nu =((\nu_x)_{x\in\Omega},\lambda_1+\lambda_2,(\nu_x^\infty)_{x\in\overline\Omega})$.
\end{lemma}
\begin{proof}
By testing $\nu_1$ with functions $\eta\otimes |\,\cdot\,|$ for $\eta\in\hold(\overline{\Omega})$, we see that
$$
|v_j|\wstar \int_\V|z|\dif\nu_\cdot(z)\lebn\restrict\Omega+\lambda_1\equiv \lambda_0\ll\lebn\quad
\text{and}\quad |\tilde v_j|\wstar \lambda_2\quad\text{in }\mathcal{M}^+(\overline\Omega).
$$
where the convergence takes place in $\mathcal{M}^+(\overline{\Omega})$. 

Let $\varepsilon\in(0,1)$. Since $\lambda_2\perp \lambda_0\ll\lebn$, we can choose an open set $O_\varepsilon$ such that
$$
\lambda_0(O_\varepsilon)<\varepsilon\lambda_0(\Omega) \quad\text{and}\quad \lambda_2(\overline{\Omega}\setminus O_\varepsilon)<\varepsilon\lambda_2(\Omega).
$$
Then for sufficiently large $j\geq j(\varepsilon)$ we have by \cite[Prop.~1.203]{FL} that 
$$
\int_{O_\varepsilon}|v_j|\leq 2\varepsilon\lambda_0(\Omega)\quad\text{ and }\quad\int_{\Omega\setminus O_\varepsilon}|\tilde v_j|<2\varepsilon \lambda_2(\Omega).
$$
In view of Lemma~\ref{lem:identify}, let $\eta\in\hold(\overline\Omega)$, $\Phi\in\LIP(\V)$ be such that $\|\eta\|_\infty\leq 1$ and $\| \nabla \Phi\|_\infty\leq 1$.
Then, for $j\geq j_\varepsilon$, write
\begin{eqnarray*}
\mathcal{E}\equiv \left|\int_{\Omega}\eta \left(\Phi(v_j+\tilde v_j)-\Phi(v_j)-\Phi(\tilde v_j)+\Phi(0)\right)\dif x\right|,
\end{eqnarray*}
which can be estimated by
\begin{eqnarray*}
\mathcal{E}&\leq& \|\eta\|_{\infty}\left(\displaystyle\int_{O_\varepsilon}|\Phi(v_j+\tilde v_j)-\Phi(\tilde v_j)|+|\Phi(v_j)-\Phi(0)|\dif x\right.\\
&&+\left.\displaystyle\int_{O_\varepsilon}|\Phi(v_j+\tilde v_j)-\Phi( v_j)|+|\Phi(\tilde v_j)-\Phi(0)|\dif x\right)\\
&\leq& 2\left(\displaystyle{\int_{O_\varepsilon}|v_j|\dif x+\int_{\Omega\setminus\varepsilon}|\tilde v_j|\dif x}\right)\leq 4\varepsilon(\lambda_0+\lambda_2)(\Omega).
\end{eqnarray*}
It follows that
\begin{eqnarray*}
\lim_{j\rightarrow\infty}\int_{\Omega}\eta \Phi(v_j+\tilde v_j)\dif x&=&\lim_{j\rightarrow\infty}\int_{\Omega}\eta\Phi(v_j)\dif x+\lim_{j\rightarrow\infty}\int_{\Omega}\eta\Phi(\tilde v_j)\dif x-\int_{\Omega}\eta\Phi(0)\dif x\\
&=&\int_{\Omega}\langle \Phi,\nu_x \rangle\dif x+\int_{\overline{\Omega}}\eta \langle \Phi^\infty \nu^\infty_x\rangle\dif (\lambda_1+\lambda_2),
\end{eqnarray*}
which concludes the proof.
\end{proof}

\section{On the angles of diffuse concentration}\label{sec:diff_conc}

As is transparent from the proof in Section~\ref{sec:inhom}, the analysis of an $\cala$-free Young measure $\nu$ is naturally split between a regular part $((\nu_x), \lambda^a\lebn,(\nu_x^\infty))$ and a singular part $((\delta_0),\lambda^s,(\nu_x^\infty))$. In this final section we will make some remarks on how the concentration (angle) measure behaves in each part.

At points of \emph{singular concentration}, by which we mean $\lambda^s$-almost everywhere, it was recently proved that the parametrized measure $(\nu_x^\infty)$ is unconstrained, in the following sense: Assuming only that the barycentre $v$ of $\nu$ is $\cala$-free, the main results of \cite{DPR,KirKri} ensure that that the Jensen-type inequality for $\cala$-quasiconvex $f$ of linear growth
\begin{align}\label{eq:jensen_sec4}
f^\infty(\overline \nu_x^\infty)\leq\int_{\mathbb{S_V}} f^\infty\dif\nu_x^\infty,
\end{align}
{holds $\lambda^s$ a.e. }$x$ due to the fact that the density $\overline\nu_\cdot^\infty$ of $v^s$ lies in $\Lambda_{\cala}$  $\lambda^s$ almost everywhere.  However, it may well happen that at points of \emph{diffuse concentration}, i.e., on a $\lambda^a\lebn$ non-negligible set, we have that $\overline\nu_x^\infty\notin\Lambda_\cala$.
In particular, at such points \eqref{eq:jensen_sec4} cannot follow from the automatic convexity result in \cite{KirKri} and, indeed,  can fail in general \cite{Mu}.

For the remainder of this section, we propose a replacement for the inequality \eqref{eq:jensen_sec4} at points of diffuse concentration. We show that this alternative inequality does indeed characterize the diffuse concentration angle measures under the technical assumption that the oscillation measure has $\cala$-free structure. We hope that this observation will be the precursor of future developments in understanding concentration effects for $\cala$-free Young measures.

By standard localization principles, e.g., Proposition~\ref{local}, it suffices to consider homogeneous Young measures. Recall the notation of Section~\ref{sec:HB} and the result of the Hahn--Banach Lemma~\ref{hb}. If $\nu=(\nu^0,t\nu^\infty) \in {Y}$ where $\nu^0\in\mathcal{M}^+_1(\mathbb V)$ and $\nu^\infty\in\mathcal{M}^+_1(\mathbb{S_V})$, then we have in particular for $\cala$-quasiconvex $f \in \HH$ that
\begin{equation}\label{testing}
f( \overline{\nu}^{0}+\overline{\nu}^{\infty}t ) \leq \int_{\V} \! f \, \dd \nu^{0} + t\int_{\SV} \! f^{\infty} \, \dd \nu^{\infty}.
\end{equation}
Now let $F \in \HH$ be $\cala$-quasiconvex and differentiable at $z_{0} \in \V$. Put for $\varepsilon > 0$
$$
f(z) = \frac{F( \varepsilon z + z_{0})-F(z_{0})}{\varepsilon}, \quad  z \in \V .
$$
Then $f \in \HH$ is $\cala$-quasiconvex and $f^{\infty}= F^{\infty}$, hence from (\ref{testing}) we get after passing to the limit $\varepsilon \searrow 0$:
$$
F^{\prime}(z_{0}) \cdot ( \overline{\nu}^{0}+\overline{\nu}^{\infty}t ) \leq \int_{\V} \! F^{\prime}(z_{0}) \cdot z \, \dd \nu^{0} + t\int_{\SV} \! F^{\infty} \, \dd \nu^{\infty}
$$
and so cancelling terms and assuming that $t>0$ we arrive at
$$
F^{\prime}(z_{0}) \cdot \overline{\nu}^{\infty} \leq \int_{\SV} \! F^{\infty} \, \dd \nu^{\infty} .
$$
If we let $D = \{ F^{\prime}( z_{0}) : \, F \mbox{ is differentiable at } z_{0} \}$, then $D$ is a bounded subset of $\V$ with $\sup_{\zeta \in D} | \zeta | = \lip (F)$.
If $G(z) = \sup_{\zeta \in D} \zeta \cdot z$ is the support function for $D$, then we have shown that the concentration angle measure $\nu^{\infty}$ must satisfy
\begin{equation}\label{concentration}
G( \overline{\nu}^{\infty}) \leq \int_{\SV} \! F^{\infty} \, \dd \nu^{\infty}.
\end{equation}
The closed convex hull $\overline{\mathrm{co}}D = \mathrm{co}\overline{D}$ is identical with the Clarke subdifferential of $F$. It is easy to check that
$G(z) = \sup_{\zeta \in \mathrm{co}\overline{D}} \zeta \cdot z$. We also record:
\begin{lemma}
Let $F \colon \V \to \R$ be of linear growth and $\cala$-quasiconvex. Then $F^{\infty} (z) \leq G(z)$ for all $z \in \V$ and equality holds at all $z \in \Lambda_{\cala}$.
Furthermore, when $F \in \HH$ is convex, then $G=F^{\infty}$ everywhere.
\end{lemma}
\begin{proof}
The inequality $F^{\infty} \leq G$ follows by the fundamental theorem and calculus. The other statements follow from the $3$-slope inequality.
\end{proof}
In particular, we retrieve \eqref{eq:jensen_sec4} at points of singular concentration without directly appealing to the result in \cite{KirKri}. At diffuse concentration points we obtain a stronger inequality than \eqref{eq:jensen_sec4}.

It would be interesting to see to what extent can the strengthened Jensen inequality \eqref{concentration} characterize (homogeneous) diffuse concentration angle measures of $\cala$-free Young measures. We will show that, under certain restrictions on the oscillation, we have a characterization.

It is not difficult to see that $(\delta_{z} , \nu^{\infty}t) \in {Y}$ for all $z \in \V$, $t>0$ if and only if $\nu^{\infty} \in \mathcal{M}^{+}_{1}( \SV )$ satisfies (\ref{concentration}) for all $\cala$-quasiconvex $F \in \HH$. In fact, we can say slightly more:

\begin{proposition}
    Let $\nu=(\nu^0,t\nu^\infty)\in\HH^*$ be  such that $F(\overline{\nu}^0)\leq \langle F,\nu^0\rangle$ for all $\cala$-quasiconvex $F \in \HH$. Then $\nu\in Y$ if and only if \eqref{concentration} holds.
\end{proposition}
This follows from the fact that for all $z\in \mathbb{V}$ we have
$$
 F(z+t\overline{\nu}^\infty)\leq F(z)+tG(\overline{\nu}^\infty),
$$
where we substitute $z=\overline\nu^0$ and use the assumed Jensen inequalities to obtain \eqref{testing}.

\bigskip

\noindent
\small {Mathematical Institute, University of Oxford, Andrew Wiles Building\\ 
Radcliffe Observatory Quarters, Woodstock Road, Oxford OX2 6GG\\ 
United Kingdom}
\\
\\
\small{\noindent Max-Planck-Institut f\"ur Mathematik in den Naturwissenschaften\\
Inselstra\ss e 22, Leipzig, 04103\\
Germany}

\end{document}